\documentclass
{amsart}
\usepackage{amscd,amsfonts,amssymb}
\usepackage{tikz}
\usepackage{pgfplots}
\usetikzlibrary{automata,positioning,calc,trees}
\usetikzlibrary{intersections,pgfplots.fillbetween}
\usepackage{graphicx,tikz}
\usepackage{pgf,tikz,pgfplots}
\usepackage{mathrsfs}
\usetikzlibrary{arrows}
\usepackage{enumerate}
\usepackage[shortlabels]{enumitem}
\usepackage{mathrsfs}
\usepackage{amssymb,amsmath,amsthm,color}
\usepackage{caption,subcaption}
\usepackage{hyperref}
\usepackage{url}
\usepackage{setspace}
\usepackage{float}

\newtheorem{thm}{Theorem}[section]

\newtheorem{lemma}[thm]{Lemma}
\newtheorem{prop}[thm]{Proposition}
\theoremstyle{definition}

\theoremstyle{remark}
\newtheorem{rem}[thm]{Remark}
\numberwithin{equation}{section}
\def\be#1 {\begin{equation} \label{#1}}
	\newcommand{\ee}{\end{equation}}

\def\sqw{\hbox{\rlap{\leavevmode\raise.3ex\hbox{$\sqcap$}}$%
		\sqcup$}}
\def\findem{\ifmmode\sqw\else{\ifhmode\unskip\fi\nobreak\hfil
		\penalty50\hskip1em\null\nobreak\hfil\sqw
		\parfillskip=0pt\finalhyphendemerits=0\endgraf}\fi}

\newcommand{\Be}{\begin{equation}}
\newcommand{\Ee}{\end{equation}}

\newcommand{\supp}{\operatorname{supp}}
\begin{document}

\title[]{Bilinear maximal functions associated with degenerate surfaces}
\author[ ]{Sanghyuk Lee and Kalachand Shuin}

\address{Department of Mathematical Sciences, Seoul National University, Seoul 08826, Republic of Korea}
\email{shklee@snu.ac.kr, kcshuin21@snu.ac.kr} 

\date{\today}

\date{\today}

\begin{abstract}

We study $L^{p}\times L^{q}\rightarrow L^{r}$-boundedness of (sub)bilinear maximal functions associated with degenerate hypersurfaces. First, we obtain the  maximal bound on the sharp range of exponents $p,q,r$  (except some border line cases) for the bilinear maximal functions given by the model surface $\big\{(y,z)\in\mathbb{R}^{n}\times \mathbb{R}^{n}:|y|^{l_{1}}+|z|^{l_{2}}=1\big\}$,  $(l_{1},l_{2})\in [1,\infty)^2$, $n\ge 2$.
 Our result  manifests  that  nonvanishing Gaussian curvature is not good enough,   in contrast with $L^p$-boundedness of the (sub)linear  maximal operator associated to hypersurfaces, to characterize the best possible maximal boundedness.    
Secondly, we consider the bilinear maximal function associated to the finite type curve in $\mathbb R^2$ and obtain a complete characterization of the maximal bound. We also prove multilinear generalizations of the aforementioned results. 
\end{abstract}
\maketitle

\section{Introduction and main results} \label{section1}

Let $d\geq2$ and  $d\sigma$ be the surface measure on a smooth hypersurface $\Sigma$. The  maximal function associated with $\Sigma$ is defined  by 
\begin{eqnarray*}
M\!f(x)=\sup_{t>0}\Big|\int_{\Sigma}f(x-ty) d\sigma(y)\Big| 
\end{eqnarray*} 
for $f\in\mathcal{S}(\mathbb{R}^{d})$.  $L^p$ boundedness of $M$ has been extensively studied for the last several decades since 
Stein's seminal work   \cite{stein}  on the spherical maximal function, which we denote by   $\mathcal M_{s}$. 
The celebrated theorem due to Stein \cite{stein} and  Bourgain \cite{Bourgain} is that 
  $\mathcal M_{s}$ is bounded on $L^{p}$ if and only if $p>{d}/(d-1)$. 
  The results naturally extends to compact surfaces with nonvanishing Gaussian curvature (see, e.g., \cite{Greenleaf}),
  which are said to be {\it nondegenerate}.  Besides $L^p$ bounds,   $L^{p}$--$L^{q}$ estimates ($p<q$) for the localized  maximal function which is defined by taking supremum over a compact interval included in $(0,\infty)$ are also almost completely understood  except for some endpoint cases  \cite{S,SS,SanghyukLee1} (see, also, \cite{ajs} and \cite{RS} for recent related developments when  the supremum is taken over on a subset of dimension less than $1$).

  The maximal functions associated  with 
  smooth degenerate surfaces were  also considered by various authors 
  (\cite{Sogge, Cowling1,Cowling2,Iosevich1,Iosevich2,Sawyer} and references therein). 
    However,  the  problem of characterizing $L^{p}$ boundedness of $M$ when Gaussian curvature of $\Sigma$  is allowed to vanish at finite order  remains largely open when $d\ge 3$.   
    The problem for curves in $\mathbb R^2$ is relatively simpler \cite{Iosevich}.   So far, this problem is better understood in dimension $d=3$ (see, e.g.,\cite{Ikromov}).  Maximal functions defined by convex surfaces   and some model surfaces were also studied 
  (\cite{Seeger, Sawyer, Iosevich1}).  There are also results for surfaces which are  non-smooth and non-convex hypersurfaces (\cite{Hong}).

\subsection*{Bilinear spherical maximal function} 

Let $n\geq1$, and let $S^{\Phi}$ denote a compact surface given by $S^{\Phi}=\{(y,z)\in \mathbb{R}^{2n}: \Phi(y,z)=0\}$ 
for a smooth function $\Phi$  on $\mathbb{R}^{2n}$.
We now consider the maximal function  
\begin{eqnarray*}
\mathfrak{M}^{{\Phi}}(f,g)(x):=\sup_{t>0}\Big|\int_{{S}^{\Phi}}f(x-ty)g(x-tz)d\sigma_{\Phi}(y,z)\Big|,
\end{eqnarray*} 
where $d\sigma_{\Phi}$ is the normalized surface measure on ${S}^{\Phi}$.  
The maximal operator can be regarded as a (sub)bilinear analogue of the (sub)linear maximal operator given by $S^{\Phi}$. 
The problem is to characterize the exponents $(p,q,r)$ for which the  estimate 
\begin{eqnarray}\label{trivialestimate}
\Vert\mathfrak{M}^{{\Phi}}(f,g)\Vert_{L^{r}}\leq C\Vert f\Vert_{L^{p}}\Vert g\Vert_{L^{q}}
\end{eqnarray}
holds. 
 When $\Phi(y,z)=|y|^{2}+|z|^{2}-1$, $\mathfrak{M}^{{\Phi}}$ is called 
 the bilinear spherical maximal function, which we denote by 
 $\mathfrak{M}_s$. Boundedness of $\mathfrak{M}_s$ was first studied by  Geba--Greenleaf--Iosevich--Palsson--Sawyer \cite{Geba}, and later 
by Barrionuevo--Grafakos--He--Honz\'{i}k--Oliveira \cite{Grafakos1} and  Heo--Hong--Yang \cite{Heo}.  They obtained partial results. 
More recently, Jeong and the first  author  \cite{Eunhee} established boundedness of $\mathfrak  M_{s}$ on the sharp  range for $n\geq2$. 
More precisely,  it was shown  that $\mathfrak{M}_s$ boundedly maps $L^{p}\times L^{q}\rightarrow L^{r}$ if ${1}/{r}={1}/{p}+{1}/{q}$
 and  
 	$$ ({1}/{p},{1}/{q})\in \mathcal{P}:=\big\{(x,y)\in [0,1)^{2}:x+y<(2n-1)/n\big\}.$$ 
	The result also has a natural multilinear generalization  (\cite{Eunhee} and \cite{Dosidis1, Dosidis2}). 
		Boundedness of the 
bilinear circular maximal function (when $n=1$)  was shown by Christ--Zhou \cite{Christ} and Dosidis--Ramos \cite{Dosidis3}, independently.  

There are also results in different directions. Sparse domination of $\mathfrak  M_{s}$ was  studied by Palsson--Sovine \cite{Pallson2} and  Borges--Foster--Ou--Pipher--Zhou \cite{ou}. 
Related results for maximal product of spherical averages were obtained  by Roncal,  Shrivastava and  the second author \cite{Luzsaurabh}.
The discrete analogues of $\mathfrak M_{s}$ were studied  by Anderson and Palsson  \cite{Anderson1,Anderson2}.  
Christ and Zhou \cite{Christ}  recently studied boundedness of bilinear lacunary maximal function defined on certain class of  curves.

However, boundedness of bilinear maximal estimate associated with other hypersurfaces than the sphere is not well understood. 
As far as the authors are aware, there are only a few results concerning bilinear maximal function associated with degenerate surfaces. 
 Chen--Grafakos--He--Honzík--Slavíková \cite{He} obtained $L^{2}\times L^{2}\rightarrow L^{1}$ estimate for bilinear maximal function defined by a $(2n-1)$-dimensional compact hypersurface which 
 has  more than   $n+2$ non-vanishing principal curvatures.  
 
 We begin our discussion, on   bilinear maximal function associated with  degenerate surfaces,  with  
 the following observation which extends the result in  \cite{He}.

\begin{prop}\label{proposition}
	Let ${S}$ be a $(2n-1)$-dimensional compact hypersurface which has at least $n+1$ non-vanishing principal curvatures. Then, the bilinear maximal function $\mathfrak{M}^{{S}}$, defined over the hypersurface ${S}$ maps $L^{p}\times L^{q}\rightarrow L^{r}$, for $1<p,q,r\leq\infty$ with ${1}/{r}={1}/{p}+{1}/{q}$.
\end{prop}

   \newcommand{\bl}{\mathbf a}
   \newcommand{\bt}{\mathbf t}

\subsection*{Bilinear maximal function associated to degenerate hypersurfaces} 
The main object of this paper is to investigate  boundedness of $\mathfrak{M}^{{\Phi}}$ with some degenerate model surfaces. 
For the purpose, we consider the surface
  \begin{eqnarray*}
S_{2}^{\bl}:=\big\{(y,z)\in\mathbb{R}^{n}\times \mathbb{R}^{n}:|y|^{\bl_{1}}+|z|^{\bl_{2}}=1 \big\}, 
\end{eqnarray*}
for  $n\geq2$ and $\bl =(\bl_1, \bl_2)\in  (0,\infty)^2$.  
%
%
We consider a biparametric bilinear averaging operator  $\mathcal{A}^{\bl}_{\bt}$ over $S_{2}^{\bl}$ which  is given by 
\begin{eqnarray*}
	\mathcal{A}^{\bl}_{\bt}(f,g)(x):=\int_{S_{2}^{\bl}}f(x-t_{1}y)g(x-t_{2}z)~d\mu(y,z), \quad \bt\in (0,\infty)^2
\end{eqnarray*}
where $d\mu$ is the normalized surface measure on $S_{2}^{\bl}$.
Instead of one parameter maximal function we study  the biparametric   maximal function
$$\mathfrak{M}^{\bl}(f,g)(x):=\sup_{\bt\in (0,\infty)^2 }\big|\mathcal{A}^{\bl}_{\bt}(f,g)(x)\big|.$$
To state our result, for $\bl\in (0,\infty)^2$ we set 
\begin{eqnarray*}
	\mathcal{P}^{\bl}&=&\Big\{(x,y)\in \mathcal{P}: x<1-\frac{(\bl_2-n)_{+}}{\bl_{2}n},\ \  y<1- \frac{(\bl_1-n)_{+}}{\bl_{1}n} \Big\},
\end{eqnarray*} 
where $a_{+}=\max\{a, 0\}$. Note that  $\mathcal{P}^{\bl}=\mathcal{P}$ when $\bl_1, \bl_2 \leq n$.

\begin{thm}\label{Theorem1}
	Let $n\geq2$ and $\bl\in [1,\infty)^2$. Let $1\leq p,q\leq\infty$ and $r>0$ with  ${1}/{r}={1}/{p}+{1}/{q}$. Then, the estimate 
	\begin{eqnarray}\label{theestimate}
\Vert \mathfrak{M}^{\bl}(f,g)\Vert_{L^{r}(\mathbb{R}^{n})}\lesssim \Vert f\Vert_{L^{p}(\mathbb{R}^{n})}\Vert g\Vert_{L^{q}(\mathbb{R}^{n})}
\end{eqnarray}
 holds if $({1}/{p},{1}/{q})\in \mathcal{P}^{\bl}$.
\end{thm}

One can also obtain  weak or restricted weak type bounds on  $\mathfrak{M}^{\bl}$ for $({1}/{p},{1}/{q})$ which belongs  to the boundary line segments of $\mathcal{P}^{\bl}$, using a simple summation trick (for example, see \cite{Eunhee} and \cite[Lemma $2.6$]{SanghyukLee1}).
However,  in this article we focus on the strong type bound.  In order to prove Theorem \ref{Theorem1}, we make use of  the slicing argument from \cite{Eunhee}, by which it was also possible to  prove boundedness of the biparametric bilinear spherical maximal function. In Section \ref{multilinearversion} we also consider multilinear maximal operators  generalizing Theorem \ref{Theorem1}.

\begin{rem}[Sharpness of Theorem \ref{Theorem1}]\label{rem1}
The range of exponents in Theorem 1.2 is sharp in that 
	the estimate \eqref{theestimate} does not hold if $({1}/{p},{1}/{q})\notin \overline{{\mathcal{P}}^{\bl}}$, where $\overline{{\mathcal{P}}^{\bl}}$ denotes the closure of  $\mathcal{P}^{\bl}$.  Moreover, the operator $\mathfrak{M}^{\bl}$ does not map $L^{1}\times L^{\infty}\rightarrow L^{1}$ and $ L^{\infty}\times L^{1}\rightarrow L^{1}$. The same holds true for the maximal operator  
	 \vspace{-3pt}
	 \Be\label{1-parameter} \widetilde {\mathfrak{M}}^{\bl}(f,g)(x):=\sup_{t\in (0,\infty) }\big|\mathcal{A}^{\bl}_{(t,t)}(f,g)(x)\big|.
	 \Ee	 
	
	 \vspace{-7pt}
\noindent (See Section \ref{sec:2.2}.) Hence, Theorem \ref{Theorem1} also holds even if  $\mathfrak{M}^{\bl}$ is replaced by $ \widetilde {\mathfrak{M}}^{\bl}	$. \end{rem}


 Consider  $S_{2}^{(2,2l_{2})}$ for an integer $l_{2}\geq2$, i.e., the hypersurface which is  given by   $y^{2}_{1}+y^{2}_{2}+\cdots +y^{2}_{n}+(z^{2}_{1}+\cdots+z^{2}_{n})^{l_{2}}=1$. Observe that Gaussian curvature of $S_{2}^{(2,2l_{2})}$ vanishes at the point $(\theta,0)\in\mathbb{R}^{n}\times \mathbb{R}^{n}$ whenever $|\theta|=1$.  However, Theorem \ref{Theorem1} implies that $\mathfrak{M}^{(2,2l_{2})}$ enjoys the same boundedness property as the bilinear spherical maximal function, that is to say, $\mathfrak{M}^{(2,2l_{2})}$ is bounded from $L^{p}\times L^{q}$ to $L^{r}$, ${1}/{r}={1}/{p}+{1}/{q}$, for 
 	$ ({1}/{p},{1}/{q})\in \mathcal{P}$ as long as $ 2l_2\leq n$.   In contrast to $L^p$ boundedness of  the linear maximal function, the boundedness of $\mathfrak{M}^{\bl}$ is less sensitive to  curvature vanishing.

\begin{figure}[t]
	\begin{tikzpicture}[scale=3.8,font=\small]]
		\fill[lightgray] (0,0)--(1,0)--(0.5,0.5)--(0,0.5)--cycle;
		\draw[densely dotted]  (0,1) -- (1,0);
		\draw[thin][->]  (0,0) --(1.15,0) node[right]{$1/p$};
		\draw[thin][->]  (0,0) --(0,1.2) node[left]{$1/q$};
		\draw[densely dotted] (0,1) node[left]{$(0,1)$} --(1,1);
		\draw [densely dotted] (1,0)node[below]{$(1,0)$} --(1,1);
		\draw[densely dotted] (0,0.5)node[left]{$(0,\frac{1}{m})$}--(0.5,0.5)node[right]{$(\frac{m-1}{m},\frac{1}{m})$};
			\end{tikzpicture}
	\begin{tikzpicture}[scale=3.8]
		\fill[lightgray] (0,0)--(1,0)--(0,0.5)--cycle;
		\draw[densely dotted]  (0,1) -- (1,0);
		\draw[thin][->]  (0,0) --(1.15,0) node[right]{$1/p$};
		\draw[thin][->]  (0,0) --(0,1.2) node[left]{$1/q$};
		\draw[densely dotted] (0,1) node[left]{$(0,1)$} --(1,1);
		\draw [densely dotted] (1,0) node[below]{$(1,0)$} --(1,1);
		\draw[densely dotted] (0,0.5)node[left]{$(0,\frac{1}{m})$}--(0.5,0.5)node[right]{$(\frac{m-1}{m},\frac{1}{m})$};
		\draw[densely dotted] (0,0.5)--(1,0);
		
	\end{tikzpicture}
\caption{The operator $\mathcal{M}^{\gamma}$ is bounded for  $(1/p,1/q)$ belonging  to the gray regions. }
\label{finitetype}
\end{figure}
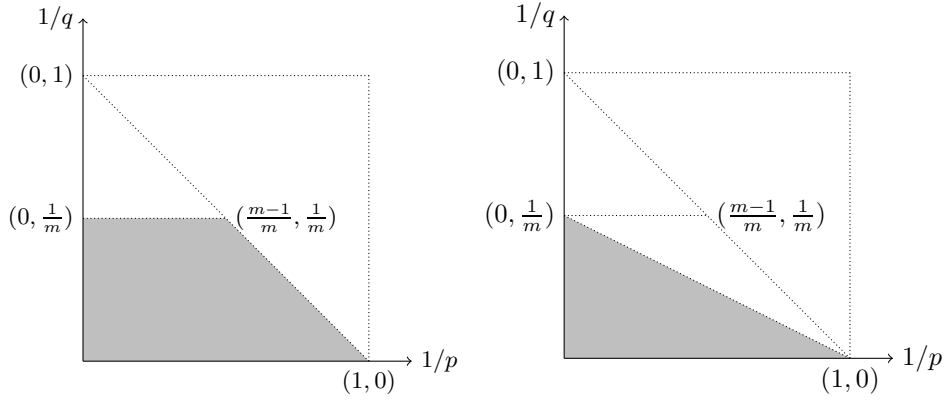

\subsection*{Bilinear maximal functions associated to curves}
Let $I$ denote the interval $(0,1)$ and let $\gamma: I\to \mathbb R^2$  be a smooth  regular curve.  We consider the maximal operator 
\[
	\mathcal{M}^{\gamma}(f,g)(x):=\sup_{\bt>0}\Big|\int f(x-t_{1}\gamma_{1}(s))g(x-t_{2}\gamma_{2}(s))\psi(s)ds\Big|, 
	\]
	where $\gamma(s)=(\gamma_{1}(s),\gamma_{2}(s))$ and $\psi\in C_c^\infty(I)$. If  $\gamma'_{1}(s)\neq0$ and $\gamma'_{2}(s)\neq0$ on $\supp \psi$,
	we have a complete characterization of $L^p\times L^q\rightarrow L^r$ boundedness.

	\begin{thm}\label{result2} 
		Let $0< p,q\leq\infty$ and ${1}/{r}={1}/{p}+{1}/{q}$. Let  $\psi\in C_c^\infty(I)$ be nonnegative and nontrivial.  Suppose $\gamma'_{1}(s)\neq0$ and $\gamma'_{2}(s)\neq0$ on $\supp \psi$.   Then,  
		\begin{eqnarray}\label{estimateforfinitetypecurve}
\Vert \mathcal{M}^{\gamma}(f,g)\Vert_{L^{r}(\mathbb{R})}\lesssim \Vert f\Vert_{L^{p}(\mathbb{R})}\Vert g\Vert_{L^{q}(\mathbb{R})}
\end{eqnarray}holds if and only if $1<r\leq\infty$. 
		\end{thm}

	Since $\gamma$ is a regular curve (i.e., $\gamma'\neq 0$), by  symmetry  and a change of variables 
	we may assume  that   $\gamma$ is of the form 
	\Be 
	\label{sform} \gamma(s)=(s,\gamma_{2}(s))
	\Ee  in a small neighborhood of $s_{0}\in I$.  For 
	$m\ge 2$, we say that $\gamma_2$ is of type $m$ at $s$ if  
	$\gamma_{2}^{(k)}(s)=0$ for $1\leq k<m$ and $\gamma_{2}^{(m)}(s)\neq0$.  
	Iosevich \cite{Iosevich} proved that  the linear maximal function  associated with  curves of type $m$ is bounded on $L^{p}(\mathbb{R}^{2})$ if and only if $p>m$. 
	The following  may be regarded as a bilinear counterpart of his result.  

\begin{thm}\label{result1}
 Let $0< p,q\leq\infty$, $m\ge 2$,  and ${1}/{r}={1}/{p}+{1}/{q}$.  Let  $\psi\in C_c^\infty(I)$ be nonnegative.
 Suppose $\gamma$ takes the form \eqref{sform} in a neighborhood of $s_0\in I$ and $\gamma_2$ is at most type $m$ at $s_0$.   
 If $\supp \psi$ is contained in a sufficiently small neighborhood of $s_0$ and $\psi(s_0)\neq 0$, then we have the following:
	\begin{itemize}
	         \item[$(i)$]  Suppose $\gamma(s_0)=0$. Then, \eqref{estimateforfinitetypecurve} holds if and only if $1<p,q,r\leq\infty$.
		\item [$(ii)$]  Suppose $\gamma(s_0)\neq 0$ and $s_0=0$. Then, \eqref{estimateforfinitetypecurve} holds if and only if $1<p\leq\infty$, $m<q\leq\infty$, and $1<r\le \infty$.
		\item [$(ii)$]  Suppose $\gamma(s_0)\neq 0$ and $s_0\neq0$. Then,  \eqref{estimateforfinitetypecurve} holds if and only if  $1<p,q\leq\infty$ and  ${m}/{q}+{1}/{p}<1$.
	\end{itemize}
\end{thm}

The assertion $(ii)$ in Theorem \ref{result1}  recovers the optimal boundedness  of the bilinear circular maximal function which was obtained in \cite{Christ,Dosidis3}.

The two figures in Figure \ref{finitetype}  depict the assertions $(ii)$ and $(iii)$ in Theorem \ref{result1}. That is to say,  suppose that $\gamma$ is of type $m$ at $s_{0}=0$ with $\gamma(s_0)\neq0$, then \eqref{estimateforfinitetypecurve} holds if and only if $({1}/{p},{1}/{q})$ belongs to the gray region of the left figure.  
Suppose that $\gamma$ is of type at $s_{0}\in (0,1]$ with $\gamma_{2}(s_{0})\neq0$, then \eqref{estimateforfinitetypecurve} holds if and only if $({1}/{p},{1}/{q})$ belongs to the triangular gray region in the right figure.

\subsection*{Notations} In what follows,   $C$ denotes a positive constant which may change at each occurrence  and depend on the dimension $n$. 
For any two positive real numbers $A$ and $B$ we write $A\lesssim B$ and $A\gtrsim B$ if there exist constants $c,c'>0$ such that $A\leq cB$ and $A\geq c' B$, respectively. By $A \sim B$ we mean that $A\lesssim B$ and $A\gtrsim B$. For a triple $(p,q,r)$, we always assume that ${1}/{r}={1}/{p}+{1}/{q}$.  

\subsection*{Organization of the paper}  In Section~\ref{sufficientpart1} we prove 
Theorem \ref{Theorem1} and Proposition \ref{proposition} and show sharpness of the results in Theorem \ref{Theorem1}. The proofs of Theorem \ref{result1} and Theorem \ref{result2} are given in Section~\ref{sufficientpart2}, where optimality of the results in Theorem \ref{result1} and Theorem \ref{result2} are also discussed.  In Section~\ref{multilinearversion}  we generalize the results in Theorem \ref{Theorem1}  into a multilinear setting. 

\section{ Proof of Theorem \ref{Theorem1} and Proposition  \ref{proposition} }\label{sufficientpart1}
In this section we prove boundedness of the bilinear maximal functions $\mathfrak{M}^{\bl}$ and $\mathfrak{M}^{\mathcal{S}}$ associated with  a compact hypersurface of having $k>1$ non-vanishing principal curvatures.
 We also show  sharpness of  the results in Theorem \ref{Theorem1} considering particular functions.

\subsection{Proof of Theorem \ref{Theorem1}} 
 To prove boundedness of $\mathfrak{M}^{\bl}$, we utilize the slicing argument from \cite{SanghyukLee1}, which 
 allows us to break the bilinear maximal operator to a product of two linear maximal operators. 
 We begin by recalling the following formula  (\cite[p.136]{Hormander}).
 Let $\Omega$ be a $(k-1)$-dimensional hypersurface in $\mathbb{R}^{k}$ given by $\Omega=\{\omega\in \mathbb{R}^{k}:\Phi(\omega)=0\}$. If $|\nabla\Phi(\omega)|\neq0$ for $\omega\in\Omega$, then 
\begin{eqnarray}\label{identity}
	\int_{\Omega}F(\omega)\frac{d\mu(\omega)}{|\nabla\Phi(\omega)|}=\int_{\mathbb{R}^{k}}F(\omega)\delta(\Phi)~d\omega,
\end{eqnarray}
where $d\mu$ is the induced surface measure on $\Omega$ and $F$ is a continuous function on $\mathbb{R}^{k}$.

We may assume that $f, g\ge 0$. 
	 Using the identity \eqref{identity} for the hypersurface $S^{\bl}_{2}=\{(y,z)\in \mathbb{R}^{n}\times \mathbb{R}^{n}:\Phi_2^\bl(y,z):=|y|^{\bl_1}+|z|^{\bl_2}-1=0\}$, we get 
	\begin{eqnarray*}
			\int_{S_{2}^{\bl}}f(x-t_{1}y)g(x-t_{2}z)~\frac{d\mu(y,z)}{|\nabla\Phi_2^\bl(y,z)|}
			 =\int_{\mathbb{R}^{n}}\int_{\mathbb{R}^{n}}f(x-t_1y)g(x-t_2z)\delta(\Phi_2^\bl)dzdy . 
		\end{eqnarray*}
	Since  $|\nabla\Phi_2^\bl(y,z)|^{2}=\bl^{2}_{1}|y|^{2(\bl_{1}-1)}+\bl^{2}_{2}|z|^{2(\bl_{2}-1)}\neq0$ and $\bl_1,\bl_2\ge 1$, we see   $c_{\bl}\leq|\nabla\Phi_2^\bl(y,z)|^{2}\leq C_{\bl}<\infty$ for all $(y,z)\in S_2^\bl$ and  some positive  constants $c_{\bl},C_{\bl}$.  Thus, it follows that 
\begin{eqnarray*}
			\mathcal{A}^{\bl}_{\bt}(f,g)(x) 
		\sim\int_{\mathbb{R}^{n}}\int_{\mathbb{R}^{n}}f(x-t_1y)g(x-t_2z)\delta(\Phi_2^\bl)dzdy . 
		\end{eqnarray*}	
Note that the left hand side is equal to	$\int_{\mathbb{R}^{n}}f(x-t_1y)(\int_{\mathbb{R}^{n}}g(x-t_2z)\delta(\Phi_2^\bl(y, \cdot)dz)dy$. Thus, fixing $y$ and using the identity \eqref{identity} with $\Phi_2^\bl(y,\cdot)$, we have  
	\begin{eqnarray*}
		\mathcal{A}^{\bl}_{\bt}(f,g)(x) 
		\sim \int_{|y|<1}f(x-t_{1}y)  \int_{\Omega_y}g(x-tz)\frac{d\mu_y^\bl (z)}{|\nabla_z\Phi_2^\bl(y,z)|}~dy,
	\end{eqnarray*}
	where $\Omega_{y}=\Phi_2^\bl(y,\cdot)^{-1}(0)\subset \mathbb{R}^{n}$ and $\mu_y^\bl$ is the induced surface measure on  $\Omega_{y}$.  	

For $\bl=(\bl_{1},\bl_{2})\in [1,\infty)^2$, 
we set 
\Be\label{weights} \omega_\bl (y) =  (1-|y|^{\bl_1})^{1/\bl_2},  \quad    \tilde \omega_\bl (y) =  (1-|y|^{\bl_2})^{1/\bl_1}. \Ee
Also, we set \Be\label{newaverage}\mathfrak Ag(x,t)=\int_{\mathbb{S}^{n-1}}g(x-t\theta)~d\theta.\Ee
Note that $\Omega_{y}$ is the $(n-1)$ dimensional sphere $\Omega_{y}$ of radius $\omega_\bl(y)$ and  $|\nabla_z\Phi_2^\bl(y,z)|=\bl_2 \omega^{\bl_2-1}_\bl(y)\neq0$. 
Therefore, by scaling  we obtain 
	\Be 
	\label{gs}	
	\mathcal{A}^{\bl}_{\bt}(f,g)(x) 
		\sim \int_{|y|<1}f(x-t_{1}y)   \mathfrak Ag(x,t_{2}\omega_{\bl}(y))  \omega_{\bl}^{{n-\bl_2}}(y) dy.  	
	\Ee
	Interchanging the roles of $f$ and $g$,  we also have  
	\Be
	\label{fs}
		\mathcal{A}^{\bl}_{\bt}(f,g)(x)\sim \int_{|z|<1}g(x-t_{2}z)    \mathfrak Af(x,t_{1}\tilde{\omega}_{\bl}(z))  \tilde \omega_\bl^{n-\bl_{1}} (z)dz.  
	\Ee

Depending on whether $\bl_1, \bl_2$ are less than or greater than the dimension $n$, we consider the following three cases, separately:
\begin{align*}
&(\mathrm A): \qquad \ \max\{\bl_1, \bl_2\}\leq n,  
\\ & (\mathrm B): \min\{\bl_1, \bl_2\}\leq n<\max\{\bl_1, \bl_2\}, 
\\ & (\mathrm C): \qquad \ n<\min\{\bl_1, \bl_2\}.
\end{align*}

	\subsection*{Case $(\mathrm A): \max\{\bl_1, \bl_2\}\leq n$} In this case, we note that $\mathcal{P}^{\bl}=\mathcal{P}$ and $\omega_{\bl}^{{n-\bl_2}}$ is bounded. Thus,  \eqref{gs} gives the  inequality  
	\Be
	\label{ggs}
	\mathfrak{M}^{\bl}(f,g)(x)\lesssim Mf(x)\mathcal M_{s}g(x),
	\Ee
	where $M$ denotes the Hardy--Littlewood  maximal function.  
	Using H\"older's inequality, by $L^{p}$-boundedness of  $M$ (for $1<p\leq\infty$) and  the  spherical maximal function $\mathcal M_{s}$ (for ${n}/({n-1})<p\leq\infty$), we see that $\mathfrak{M}^{\bl}$ maps $L^{p}\times L^{q}\rightarrow L^{r}$ for $({1}/{p},{1}/{q})\in [0,1)\times [0,\frac{n-1}{n})$. 
	Similarly, since $\bl_{1}\leq n$,  from \eqref{fs} we have
\Be
	\label{ffs} \mathfrak{M}^{\bl}(f,g)(x)\lesssim \mathcal M_{s}f(x)Mg(x).
	\Ee
	Therefore, H\"older's inequality and  boundedness of $M$ and $\mathcal M_{s}$ show that $\mathfrak{M}^{\bl}$ is bounded from $L^{p}\times L^{q}\rightarrow L^{r}$, for $({1}/{p},{1}/{q})\in [0,\frac{n-1}{n})\times  [0,1)$.  
	Hence, interpolation (for example, see \cite{Grafakosmodern}) gives $L^{p}\times L^{q}\rightarrow L^{r}$ bound on $\mathfrak{M}^{\bl}$ for all $({1}/{p},{1}/{q})\in \mathcal{P}$.

	\subsection*{Case $(\mathrm B): \min\{\bl_1, \bl_2\}\leq n<\max\{\bl_1, \bl_2\}$} 
	 We only consider the case $\bl_{1}\leq n<\bl_{2}$. The other case $(\bl_{2}\leq n<\bl_{1})$ can be handled  in the same manner by interchanging the roles of $\bl_{1}$ and $\bl_{2}$. We first denote some vertices of  $\mathcal{P}^{\bl}$ as follows:
	 \[  H=(0,1), \ P=\Big(\frac{n-1}{n},1\Big),\ Y=\Big(1-\frac{1}{n}+\frac{1}{\bl_{2}}, 1-\frac{1}{\bl_{2}}\Big),\  B=\Big(1-\frac{1}{n}+\frac{1}{\bl_{2}},0\Big).\]  
	 We also set $Q=(1,\frac{n-1}{n})$. (See  Figure \ref{caseB}.)

	 	\begin{figure}[t]
		\begin{tikzpicture}[scale=4, font=\small]
				\fill[lightgray](0,0)--(0.8,0)--(0.8,0.7)--(0.5,1)--(0,1)--cycle;
			\draw[densely dotted] (0,1)node[left]{$H$}--(1,1);
			\draw[densely dotted] (1,0)--(1,1);
			\draw[thin][->]  (0,0) --(1.2,0) node[right]{$1/p$};
			\draw[black,thin][->]  (0,0) --(0,1.2) node[left]{$1/q$};
			\draw[densely dotted] (0.5,0)
			--(0.5,1)node[above]{$P$};
			\draw[densely dotted] (0,0.5)--(1,0.5)node[right]{$Q$};
			\draw[densely dotted] (0.8,0)node[below]{$B$}--(0.8,1);
			\draw[densely dotted] (0.5,1)--(1,0.5);
			\draw[densely dotted] (0.8,0)--(0.8,0.7) node[right]{$Y$};
		\end{tikzpicture}
	\caption{The set ${\mathcal{P}}^{\bl}$  for  $\bl_{1}\leq n<\bl_{2}$.}
	\label{caseB}
	\end{figure}
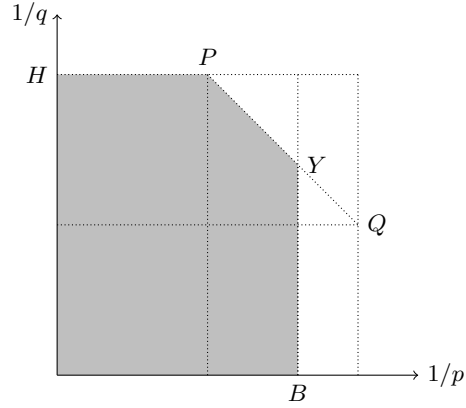

	Since $\bl_{2}>n$,  the weight $ \omega_\bl^{n-\bl_{2}} $ in \eqref{gs} is not bounded. We decompose the average away from the unit sphere. 
	For $k\ge 1$, let us set
		$\chi_k^{a}(y)=\chi_{[2^{-k}, 2^{1-k}]}(1-|y|^a)$. Then, using \eqref{gs}, we have
	 	\begin{eqnarray}
		\label{sum}
	\mathcal{A}^{\bl}_{\bt}(f,g)(x) \sim \sum_{k\ge 1} \mathcal A_{\bt}^{k}(f,g)(x),
				\end{eqnarray}
	where 
		
		$$\mathcal A_{\bt}^{k}(f,g)(x)=  \int  f(x-t_{1}y)    \chi_k^{\bl_1}(y)  \omega_\bl^{n-\bl_{2}}(y)    \mathfrak Ag(x,t_{2}\omega_{\bl}(y)) ~ dy.$$
	Since $ \chi_k^{\bl_1}(y)  \omega_\bl^{n-\bl_{2}} \le 2^{k(\bl_{2}-n)/{\bl_2}}\chi_k^{\bl_1},$ we get 
	\Be \label{mgs}
		\mathcal M^k(f,g)(x):=\sup_{\bt\in (0,\infty)^2 } |\mathcal A_{\bt}^{k}(f,g)(x)|
		 \lesssim  2^{k(\bl_{2}-n)/\bl_{2}}Mf(x)\mathcal M_{s}g(x). 
	\Ee

For $(y,z)\in S_{2}^{\mathbf a}$, $1-|y|^{\bl_1}\in [2^{-k}, 2^{1-k}]$ if and only if $|z|^{\bl_2}\in [2^{-k}, 2^{1-k}]$. 
	Therefore, interchanging the roles of  $f$ and $g$ (cf., \eqref{fs}) we see that $\mathcal A_{\bt}^{k}(f,g)(x)$ equals a constant times 
	\Be
	\label{afs}
		 \int_{2^{-k}\leq |z|^{\bl_2}\leq 2^{(1-k)}}g(x-t_{2}z)\tilde \omega_\bl^{{n-\bl_1}}(z)\mathfrak Af(x, t_{1}\tilde\omega(z))\, dz.
		\Ee
	Since $\bl_{1}\leq n$, $\tilde \omega_\bl^{({n-\bl_{1}})/{\bl_{1}}}(z)\leq 1$. So, this gives  
	\Be \label{mfs}
		\mathcal{M}^{k}(f,g)(x) \lesssim  2^{-kn/\bl_{2}}\mathcal M_{s} f(x) Mg(x). 
	\Ee
	Now, using \eqref{mgs}, H\"older's inequality  and  boundedness of $M$ and $\mathcal M_{s}$ yield  
	\[  \| \mathcal{M}^{k}(f,g)\|_r  \lesssim  2^{k(\bl_{2}-n)/\bl_{2}}  \|f\|_p\|g\|_q\] 
	for  $({1}/{p},{1}/{q})\in [0,1)\times [0,\frac{n-1}{n})$. Similarly, from \eqref{mfs}  we get   
	\[  \| \mathcal{M}^{k}(f,g)\|_r  \lesssim 2^{-kn/\bl_{2}}   \|f\|_p\|g\|_q\] 
 for $({1}/{p},{1}/{q})\in [0,\frac{n-1}{n})\times [0,1)$. Interpolation gives 
$  \| \mathcal{M}^{k}(f,g)\|_r  \lesssim 2^{-\epsilon k}   \|f\|_p\|g\|_q$ for some $\epsilon>0$ 
 if $({1}/{p},{1}/{q})$ is contained in $\mathcal P^\bl$ (i.e.,  the gray region in Figure \ref{caseB} except the boundary line segments $[H,P],$ $[P,Y]$, and $[B,Y]$). 
 Since  $\mathfrak{M}^{\bl}(f,g)\lesssim  \sum_{k\ge 1}  \mathcal{M}^{k}(f,g)$, we get the desired estimate for $\mathfrak{M}^{\bl}$ for $(1/p, 1/q)\in \mathcal P^\bl$.

\begin{rem} When $n\ge 3$, it is possible to show the weaker  $L^{p,1}\times L^{q,1} \to L^{r,\infty}$ bound for $(1/p,1/q)$ on the line segments $[H,P],$ $[P,Y]$, and $[B,Y]$, making use of the 
weak $L^1$ bound on $M$, the endpoint estimate  $L^{{n}/({n-1}),1}\to L^{{n}/({n-1}), \infty}$ for $\mathcal M_{s}$ which is due to Bourgain (\cite{Eunhee} and {\cite[Lemma 2.6]{SanghyukLee1}}).  \end{rem}

	\subsection*{Case $(\mathrm C):  n<\min\{\bl_1, \bl_2\}$}  The desired estimates in this case are less straightforward than those in the previous cases.  The hypersurface $S_{2}^{\bl}$ becomes more degenerate as $\bl_1$ and $\bl_2$ increase. 
	We now distinguish the two cases 
	\Be  
	\label{subcase}
	{1}/{\bl_{1}}+{1}/{\bl_{2}}> {1}/{n},
	\Ee and ${1}/{\bl_{1}}+{1}/{\bl_{2}}\le  {1}/{n}$. 
	We provide the proof only for the former case.  
	For the latter case,  $\mathcal P^\bl$ becomes a rectangle  and does not intersect the line segment $[P, Q]$ (see Figure \ref{caseB}). 
    However,  there is no difference between the proofs of the two cases.  For the rest of the proof, we assume 
    that \eqref{subcase} holds.

		\begin{figure}[t]
		\begin{tikzpicture}[scale=4, font=\small]
			\draw[densely dotted] (0,1)node[left]{$H$}--(1,1);
			\draw[densely dotted] (1,0)node[below]{$G$}--(1,1);
			\fill[lightgray] (0,0)--(0.8,0)--(0.8,0.7)--(0.7,0.8)--(0,0.8)--cycle;
			\draw[thin][->]  (0,0)node[left]{$O$} --(1.2,0) node[right]{$1/p$};
			\draw[thin][->]  (0,0) --(0,1.2) node[left]{$1/q$};
			\draw(0.5,1)node[above]{$P$};
			\draw(1,0.5)node[right]{$Q$};
			\draw[densely dotted] (0.8,0)node[below]{$B$}--(0.8,0.7);
			\draw[densely dotted] (0.5,1)--(1,0.5);
			\draw(0,0.8)node[left]{$A$};
			\draw(0.8,0.7) node[right]{$Y$};
			\draw[densely dotted] (0,0.8)--(0.7,0.8)node[above]{$X$};
		\end{tikzpicture}
	\caption{The set $\mathcal{P}^{\bl}$ when  $\bl_1, \bl_2 >n$ and ${1}/{\bl_{1}}+1/{\bl_{2}}> 1/{n}$.}
\label{caseC}
	\end{figure}
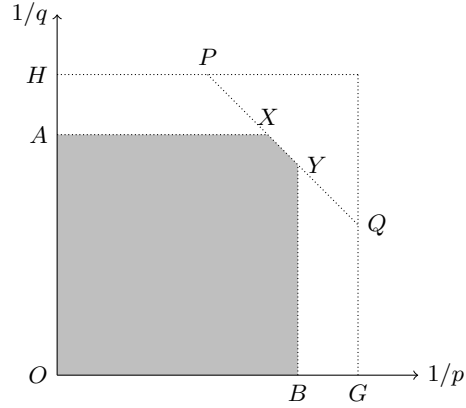

	As before, we start by denoting  some vertices of  ${\mathcal{P}}^{\bl}$: 
\[ O=(0,0), \  A=\Big(0,1-\frac{1}{n}+\frac{1}{\bl_{1}}\Big),   \ 
		X=\Big(1-\frac{1}{\bl_{1}},1-\frac{1}{n}+\frac{1}{\bl_{1}}\Big). \] 
		We also set $ G=(1,0).$
		Using the decomposition \eqref{sum}, 
we have \eqref{mgs}. Since $\bl_{1} > n$,  $\tilde \omega_\bl^{({n-\bl_{1}})/{\bl_{1}}}(z)$ is no longer bounded.
However, $\tilde \omega_\bl^{({n-\bl_{1}})/{\bl_{1}}}(z)\le C$ if $|z|\le 2^{(1-k)/\bl_2}$ for $k\ge 2$.  From \eqref{afs}  
we see \eqref{mfs} holds for $k\ge 2$.  Thus, by the same argument as before, we have  
$  \| \mathcal{M}^{k}(f,g)\|_r  \lesssim 2^{-\epsilon k}   \|f\|_p\|g\|_q$, $k\ge 2$ for some $\epsilon>0$ as long as 
$({1}/{p},{1}/{q})$ is contained in  the closed pentagon $\mathfrak P$ with vertices $O, H, P, Y, B$ from which the segments $[H,P]$, $[P,Y]$, and $[B,Y]$ are removed. So 
we have
 \[ \textstyle \|  \sum_{k\ge 2}  \mathcal{M}^{k}(f,g)\|_r\lesssim \|f\|_p\|g\|_q.\]
 if $({1}/{p},{1}/{q})$ is contained in $\mathfrak P$, which contains  $\mathcal P^\bl$.

Therefore,  we only need to consider $\mathcal M^1$.  By  \eqref{mgs} with $k=1$, we see that $\mathcal M^1$ is bounded from  $L^{p}\times L^{q}$ to $L^{r}$ for $({1}/{p},{1}/{q})\in [0,1)\times [0,\frac{n-1}{n})$. Concerning $\mathcal A_{\bt}^{1}$, from  \eqref{afs}, we note that $|z|^{\bl_2}\in [2^{-1},1]$ if and only if $1-|z|^{\bl_2}\in [0, 2^{-1}]$. Thus, we can decompose  
	\Be
	\mathcal A_{\bt}^{1}(f,g)(x)=    \sum_{j\ge 2}  \tilde{\mathcal A}_{\bt}^{j}(f,g)(x),
		\Ee
		where 
		\[  \tilde{\mathcal A}_{\bt}^{j}(f,g)(x)=\int g(x-t_{2}z)  \chi_j^{\bl_2}(z)\tilde \omega_\bl^{{n-\bl_1}}(z) \mathfrak Af(x,t_{1}\tilde{\omega}_{\bl}(z))dz.\] 

Since $ \chi_j^{\bl_2}(z)  \tilde \omega_\bl^{({n-\bl_{1}})/{\bl_{1}}}(z) \le 2^{j(\bl_{1}-n)/{\bl_1}}\chi_j^{\bl_2},$ 
we have 
\Be \label{eqn9}
	\tilde{\mathcal M}^j (f,g)(x)	:=\sup_{\bt\in (0,\infty)^2} \tilde{\mathcal A}_{\bt}^{j}(f,g)(x)  \leq C2^{j(\bl_{1}-n)/\bl_{1}}  \mathcal M_{s}f(x) Mg(x). 
\Ee 
 For $(y,z)\in S_{2}^\bl$, $1-|z|^{\bl_2}\in [2^{-j}, 2^{1-j}]$ is equivalent to $|y|^{\bl_1}\in [2^{-j}, 2^{1-j}]$. 
As before, interchanging the roles of $f$ and $g$, we see that  $\tilde{\mathcal A}_{\bt}^{j}(f,g)(x)$ is bounded above by a constant times 
\[   \int_{ 2^{-j}\le |y|^{\bl_1} \le 2^{1-j}}   f(x-t_{1}y)     \omega_\bl^{n-\bl_{2}}(y)   \mathfrak Ag(x,t_{2}\omega_{\bl}(y))    dy.\]
Note $ \omega_\bl^{n-\bl_{2}}(y)\le C$ if $|y|^{\bl_1} \le 2^{1-j}$ and $j\ge 2$. Thus, we obtain 
	\begin{eqnarray} \label{eqn10}
	\tilde{\mathcal M}^j (f,g)(x)		\leq C2^{-jn/\bl_{1}}Mf(x)\mathcal M_{s}g(x).	\end{eqnarray} 
Thus, by the same argument as before we have   $\|\tilde{\mathcal{M}}^{j}(f,g)\|_r\le  C2^{-jn/\bl_{1}}\|f\|_p\|g\|_q$   for $({1}/{p},{1}/{q})\in [0,1)\times [0,\frac{n-1}{n})$.
Similarly, using \eqref{eqn9}, we also have  $\|\tilde{\mathcal{M}}^{j}(f,g)\|_r\le C2^{j(\bl_{1}-n)/\bl_{1}} \|f\|_p\|g\|_q$  for $({1}/{p},{1}/{q})\in [0,\frac{n-1}{n})\times  [0,1)$. 
	Hence, by those bounds and interpolation, it follows that   the operator $\sum_{j\geq2}\tilde{\mathcal{M}}^{j}$ boundedly maps $L^{p}\times L^{q}\rightarrow L^{r}$ if $({1}/{p},{1}/{q})$ belongs to the closed pentagon $\tilde{\mathfrak P}$ with vertices $O,A,X,Q, G$ from which the segments $[A,X]$, $[X,Q]$, and $[Q,G]$ are removed (see Figure \ref{caseC}).  
	
	 Therefore,  \eqref{theestimate}  holds for $({1}/{p},{1}/{q})\in {\mathfrak P}\cap \tilde{\mathfrak P}$, which is equal to $\mathcal{P}^{\bl}$.

\subsection{Sharpness of Theorem \ref{Theorem1}}
\label{sec:2.2}
In this section, as mentioned in Remark \ref{rem1}, we show that the range of $(p,q)$ for the estimate  \eqref{theestimate}  in Theorem \ref{Theorem1}   is sharp except for some borderline cases. 
That is to say, we  prove that   \eqref{theestimate}  implies 
\begin{align}
\label{nec1}		&\frac{1}{r}=\frac{1}{p}+\frac{1}{q}\leq \frac{2n-1}{n}, 
		\\
\label{nec2} 		&\frac{1}{p}\leq 1-\Big(\frac{1}{n}-\frac{1}{\bl_{2}}\Big),   \quad \bl_2 > n,
		\\
\label{nec3} 		& \frac{1}{q}\leq 1-\Big(\frac{1}{n}-\frac{1}{\bl_{1}}\Big),  \quad \bl_1 > n. 
		\end{align}
In fact, we show that  
the estimate 
\begin{eqnarray}\label{theestimate1}
\Vert \widetilde{\mathfrak{M}}^{\bl}(f,g)\Vert_{L^{r}(\mathbb{R}^{n})}\lesssim \Vert f\Vert_{L^{p}(\mathbb{R}^{n})}\Vert g\Vert_{L^{q}(\mathbb{R}^{n})}
\end{eqnarray}
implies \eqref{nec1}, \eqref{nec2}, and \eqref{nec3} (see\eqref{1-parameter}~for the  definition of $\widetilde{\mathfrak{M}}^{\bl}$).

We first show \eqref{nec1}.  Set $h(s)=s^{\bl_{1}}+s^{\bl_{2}}-1$ for $s\ge 0$ and let  $s_\ast$ denote the unique point such that 
 $h(s_\ast)=0$ and $s_\ast\in (0,1)$. 
	Taking  a small positive number $\epsilon_0$ such that $s_\ast+\epsilon_0<1$, 
		we set 
	\[f=\chi_{B(0,2\delta/s_\ast)}, \quad g=\chi_{B(0,C_{1}\delta)}\] for $0<\delta\ll\epsilon_{0}$ and a constant $C_{1}>0$.
	Taking $t_{1}=t_{2}={|x|}/{s_\ast}$,  from  \eqref{gs} we have
	\[
		\widetilde{\mathfrak M}^\bl(f,g)(x)\gtrsim \int_{|y|\leq1}f\Big(x-\frac{|x|}{s_\ast}y\Big)\omega_\bl^{n-\bl_{2}}(y) \mathfrak Ag\Big(x,\frac{|x|}{s_{*}}\omega_{\bl}(y)\Big)~
	dy. 
	\] 
	We consider $
		\mathrm A=\{x\in\mathbb{R}^{n}:s_\ast\leq |x|\leq s_\ast+\epsilon_{0}\}.$
	If $|x-(s_\ast^{-1}|x|)y |\le \delta $ for  $x\in \mathrm A$,  $|s_\ast-|y||\le \delta$. Thus, recalling $\omega_\bl(y)=(1-|y|^{\bl_{1}})^{1/{\bl_{2}}}$, we see that $\omega_\bl(y)$
	in the integral is $\sim 1$ if $\delta$ is small enough. 
		Besides, our choice of $s_\ast$ ensures that  $|s_\ast-\omega_\bl(y)|=|(1-s_\ast^{\bl_1})^{1/\bl_2}-(1-|y|^{\bl_{1}})^{1/{\bl_{2}}}|\le C_2\delta$ for some constant $C_2$. 
		Therefore, for $x\in A$,   we have
	\[
		\widetilde{\mathfrak M}^\bl(f,g)(x)\gtrsim \int_{\{ |y|\leq1, |y-|x|^{-1} s_\ast x|\le\delta\}} 
		 \int_{\mathbb{S}^{n-1}\cap\{\theta:  ||x|^{-1} x- \theta|\le C\delta\} } d\sigma(\theta)dy
	\]
	for a constant $	C>0$ large enough. This gives $\widetilde{\mathfrak M}^\bl(f,g)(x)\gtrsim  \delta^{2n-1}$ for $x\in \mathrm A$.  Thus, the estimate \eqref{theestimate1} implies that $\delta^{2n-1}\lesssim \delta^{n({1}/{p}+{1}/{q})}$. Therefore, taking $\delta\rightarrow0$ we get 
	\eqref{nec1}

We now show \eqref{nec2}.  For $0<\delta<1/2$ we take  $f=\chi_{B(0,C\delta)}$ with a constant $C>2$ and $g=\chi_{B(0,10)}$. 
From \eqref{fs}, taking $t_{1}=t_{2}=|x|$, we have  
	\[
		\widetilde{\mathfrak M}^\bl(f,g)(x)
			\geq \int_{|z|^{\bl_{2}}\leq\delta}g(x-|x|z) \tilde \omega_\bl^{n-\bl_{1}} (z) \mathfrak Af(x,|x|\tilde{\omega}_{\bl}(z))dz. 
	\]
Recalling \eqref{weights}, note that	$\tilde \omega_\bl(z)\sim 1$  and $|\tilde{\omega}_\bl(z) -1|\lesssim \delta$ in the integral since $|z|^{\bl_{2}}\leq\delta$.  Thus, $\widetilde{\mathfrak M}^{\bl}(f,g)(x)\gtrsim \delta^{\frac{n}{\bl_{2}}+n-1}$ for $1\leq |x|\leq2$. This and \eqref{theestimate1} yield $\delta^{\frac{n}{\bl_{2}}+n-1}\lesssim \delta^{\frac{n}{p}}$. Therefore, taking $\delta\rightarrow0$, we get   \eqref{nec2}. 
			In a similar way one can also obtain \eqref{nec3}. 

	To show  that $\widetilde{\mathfrak M}^{\bl}$ does not boundedly map $L^{1}\times L^{\infty}\rightarrow L^{1}$ and $L^{\infty}\times L^{1}\rightarrow L^{1}$, it is enough to consider the case $\bl_1, \bl_2 \leq n$ thanks to 
	\eqref{nec2} and \eqref{nec3}. 
	Let $f\in L^{1}$ and $g(x)=1$. Using \eqref{gs}, we have 
	\[
		\mathcal{A}^{\bl}_{\bt}(f,g)(x)		\gtrsim \int_{|y|\leq1}f(x-t_{1}y)(1-|y|^{\bl_{1}})^{\frac{n-\bl_{2}}{\bl_{2}}}~dy.
	\]
	Thus, $\widetilde{\mathfrak M}^{\bl}(f,g)(x)\gtrsim Mf(x)$. Since the Hardy--Littlewood maximal function $M$ does not map $L^{1}\rightarrow L^{1}$, this implies that $\widetilde{\mathfrak M}^{\bl}$ fails to map $L^{1}\times L^{\infty}\rightarrow L^{1}$. Similarly, interchanging the roles of $f$ and $g$, one sees  that $\widetilde{\mathfrak M}^{\bl}$ does not map $L^{\infty}\times L^{1}\rightarrow L^{1}$.

\subsection{Proof of Proposition  \ref{proposition}}
In order to prove Proposition  \ref{proposition}, we make use of the following theorems. 

\begin{thm}[\cite{Littman}]\label{Littmanlemma}
	Let $\mathcal{S}$ be a compact surface which has $k\geq1$ non-vanishing principal curvatures at every point, and let $d\sigma_{\mathcal{S}}$ be a smooth surface measure on $\mathcal{S}$. Then, for all multi-indices $\alpha$, we have
	$$|\partial^{\alpha}\widehat{d\sigma_{\mathcal{S}}}(\xi)|\leq C_{\alpha}(1+|\xi|)^{-k/2}.$$
\end{thm}

Let $\{T_{t}\}_{t>0}$ be a family of Fourier multiplier operators which are defined by  $\widehat{T_{t}f}(\xi)=\hat{f}(\xi)m(t\xi)$ for  $m\in L^{\infty}(\mathbb{R}^{n})$. 
\begin{thm}[{\cite[Theorem B]{Rubiodefrancia}}]\label{Rubiotheorem}
	Let $s$ be an integer greater than ${n}/{2}$ such that $m\in C^{s+1}(\mathbb{R}^{n})$ and $|D^{\alpha}m(\xi)|\leq C|\xi|^{-a}$ for $|\alpha|\leq s+1$ and some  $a>1/2$. 
	Then, the maximal operator $T^{*}f(x)=\sup_{t>0} |T_{t}f(x)|$ is bounded on  $L^{p}$ if \[p_{0}:=2n/(n+2a-1)<p<p_{\infty}:=(2n-2)/({n-2a}),\] where we set $p_{0}=1$ if $a\geq \frac{n+1}{2}$ and $p_{\infty}=\infty$ if $a\geq n/2$.
\end{thm}

	We begin with claiming  that the maximal function $\mathfrak{M}^{{S}}$  maps $L^{p}\times L^{\infty}\rightarrow L^{p}$, for $1<p\leq\infty$. The idea here is similar to the one in \cite{Grafakos1}. For $g\in L^{\infty}$, we have 
	\begin{eqnarray*}
		\mathfrak{M}^{{S}}(f,g)(x)\leq \Vert g\Vert_{L^{\infty}}\sup_{t>0}\Big|\int_{{S}}f(x-ty)~d\sigma(y,z)\Big|,
	\end{eqnarray*}
	where $d\sigma$ is the normalized surface measure on the surface ${S}$. By  Fourier inversion  we see
	\begin{eqnarray*}
		\int_{{S}}f(x-ty)\,d\sigma_{{S}}(y,z)
		=\int_{\mathbb{R}^{n}}\widehat{d\sigma}(t\xi,0)\widehat{f}(\xi) e^{2\pi\iota x\cdot\xi}\,d\xi,
	\end{eqnarray*}
	where  $\widehat{d\sigma}$ denotes the Fourier transform of  $d\sigma$.
	Now, we set $m(\xi)=\widehat{d\sigma}(\xi,0)$. Since $S$ has at least $ n+1$ nonvanishing  principal curvatures,  $m$ satisfies  $|\partial^{\alpha}m(\xi)|\leq C|\xi|^{-a}$ with $a=(n+1)/2$ by Theorem \ref{Littmanlemma}. 
	Using Theorem \ref{Rubiotheorem}, we see that $T^\ast$ is bounded on  $L^{p}$  for $1<p<\infty$.  Thus, it follows that $\mathfrak{M}^{{S}}$ is bounded from $L^{p}\times L^{\infty}$ to $L^{p}$ for $p>1$. Similarly,  by interchanging the roles of $f$ and $g$ it follows that $\mathfrak{M}^{{S}}$ maps $L^{\infty}\times L^{q}\rightarrow L^{q}$ for $q>1$. Therefore, interpolation completes the proof.

	\section{Proof of Theorem \ref{result2} and \ref{result1}}\label{sufficientpart2}
	
In this section, we prove  Theorem \ref{result2} and  \ref{result1}.  Sharpness of the ranges of $p,q$ in the theorems is shown by considering particular examples.

	\newcommand{\pn}{{L^p}}
	\newcommand{\rn}{{L^r}}
	\newcommand{\qn}{{L^q}}

\subsection{Proof of Theorem \ref{result2}}  
We first prove the sufficiency part of Theorem \ref{result2}. 
 Since $\gamma_2'(s)\neq 0$ on $\supp \psi$, changing variables $y=\gamma_2(s)$, we have 
\[
	\mathcal{M}^{\gamma}(f,g)(x)
		\le \Vert f\Vert_{L^{\infty}}\sup_{t_{2}>0}\int^{1}_{0} |g(x-t_{2}\gamma_{2}(s))\psi(s)| ds  \lesssim \Vert f\Vert_{L^{\infty}} Mg(x).
\]
Thus, by the Hardy--Littlewood maximal  inequality we have 
\Be\label{est-p0}
	\|\mathcal{M}^{\gamma}(f,g)\|_\qn  \lesssim \Vert f\Vert_{L^{\infty}} \|g\|_\qn
	\Ee
for $1<q\leq\infty$.
Similarly, since  $\gamma'_{1}(s)\neq0$  on $\supp \psi$, we get 
$
	\mathcal{M}^{\gamma}(f,g)(x)
	\lesssim Mf(x) \Vert g\Vert_{L^{\infty}} .
$
This gives 
\Be\label{est-q0} \|\mathcal{M}^{\gamma}(f,g)\|_\pn  \lesssim  \|f\|_\pn \Vert g\Vert_{L^{\infty}} \Ee
 for $1<p\leq\infty$.
By interpolation we obtain the desired $L^p\times L^q\to L^r$ boundedness for $1< p,q\leq\infty$ and ${1}/{r}={1}/{p}+{1}/{q}$.

We now  show  the necessity part. That is to say,  $\mathcal{M}^{\gamma}$ can not be bounded from $L^{p}\times L^{q}\rightarrow L^{r}$ if $r\leq1$. 
Since $\gamma'_{1}(s)\neq 0$ and $\gamma'_{2}(s)\neq0$  on $\supp \psi$, there is a point $s_0$ 
such that $\gamma_{1}(s_{0})\neq 0$, $\gamma_{2}(s_{0})\neq0$, and $\psi(s_0)> 0$. 
Thus,  we have 
\Be\label{compare} |\gamma_{1}(s_{0})-\gamma_{1}(s)| \sim |s-s_0|, \quad |\gamma_{2}(s_{0})-\gamma_{2}(s)|\sim |s-s_0|\Ee 
 if   $s\in [s_{0},s_{0}+\delta]$  with a constant $\delta>0$ small enough. We may also assume $\psi\gtrsim 1$ on the interval $ [s_{0},s_{0}+\delta]$. 
		We consider 
		\[ f(x)=|x|^{-1/p}\log^{-2/p}(|x|^{-1})\chi_{(-1,1)}, \quad  g(x)=|x|^{-1/q}\log^{-2/q}(|x|^{-1})\chi_{(-1,1)}.\]
		 Then,  it is easy to see that $f\in L^{p}(\mathbb{R})$ and $g\in L^{q}(\mathbb{R})$ for $0<p,q<\infty$.  Taking $t_1={x}/{\gamma_{1}(s_{0})}$ and $t_2={x}/{\gamma_{2}(s_{0})}$, 		 
		  we  have
		\Be 
		\label{max-}		\mathcal{M}^{\gamma}(f,g)(x)
			\gtrsim \int^{s_{0}+\delta}_{s_{0}} f\big(x-\frac{x\gamma_{1}(s)}{\gamma_{1}(s_{0})}\big)g\big(x-\frac{x\gamma_{2}(s)}{\gamma_{2}(s_{0})}\big)ds.\Ee
			Thus,  using \eqref{compare}, we have
			\[  \mathcal{M}^{\gamma}(f,g)(x)\gtrsim  \int^{s_{0}+c}_{s_{0}} |x(s_{0}-s)|^{-\frac1p-\frac1q}\log^{-\frac2p-\frac2q}\big(\frac{c}{x|s_{0}-s|}\big)
						ds \]
		for  $x\in(0, 1)$ and a small $c>0$. 
		Therefore,  after a change of variables  we get
		\begin{eqnarray*}
			\mathcal{M}^{\gamma}(f,g)(x)
			\gtrsim  x^{-1}\int^{cx}_{0}t^{-1/r}\log^{-2/r}(c/t)~dt\gtrsim   
			\begin{cases}
				 \,(x\log(c/x))^{-1}, &\ \text{if}\quad r=1,\\
				\qquad \infty, &\  \text{if}\quad r<1.
			\end{cases}
		\end{eqnarray*}
		So, it follows that  $\|\mathcal{M}^{\gamma}(f,g)\|_\rn=\infty$ if $r\le 1$ as desired.  
		When $p=\infty$ ($q=\infty$), taking $f=\chi_{(-1,1)}$ ($g=\chi_{(-1,1)}$, respectively), we get the same conclusion. 
		

	\begin{rem} 
	\label{rem-nec}
	By the same argument in the above showing the necessity part  we can also show that the maximal estimate \eqref{estimateforfinitetypecurve} fails if  $1/p+1/q=1/r\geq 1$ when 
	$\gamma(s)=(s,\gamma_{2}(s))$ and $\gamma_{2}$ is of type $m$ at $s_0$. The argument remains valid as long as we can 
	find an interval $[s_\ast, s_\ast+\delta]$ such that $\gamma'_{2}(s)\neq0$ on  $[s_\ast, s_\ast+\delta]$ for a small $\delta>0$ and $\psi>0$ on the interval $[s_\ast, s_\ast+\delta]$. 
	One can clearly find  such an interval, possibly, away from $s_0$ (see \eqref{near} below). 		
	\end{rem}	

	\newcommand{\ninf}{{L^\infty}}

	\subsection{Proof of Theorem \ref{result1}}
	Since $\gamma_{2}$ is of type $m$ at $s_0$, i.e.,  $\gamma'_{2}(s_{0})=\gamma''_{2}(s_{0})=\cdots=\gamma^{(m-1)}_{2}(s_{0})=0$ and  $\gamma^{(m)}_{2}(s_{0})\neq0$ in a small neighborhood of $s_0$, we have  
		\Be\label{near} 
		\gamma_{2}(s)=\gamma_{2}(s_{0})-(s-s_0)^{m}\phi(s),  \quad s\in  I_0:= [s_0-c, s_0+c], 
		\Ee
		for some $c>0$ where  
		$\phi$ is smooth and  $|\phi|>c_0$  on $I_0$ for a constant $c_0>0$. 
		We may assume  $\supp \psi\subset I_0$ by taking $\supp \psi$ to be small enough.

	\subsubsection{Proof of  $(i)$
	}  
	It is clear that  $\mathcal{M}^{\gamma}(f,g)(x)\lesssim Mf(x)\|g\|_{L^\infty}$. This gives \eqref{est-q0} 	
	for $1<p\le \infty$.  Since $\gamma_{2}(s_{0})=0$, changing variables $y=(s-s_0)^{m}\phi(s)$, we have  
	\begin{eqnarray*}
		\mathcal{M}^\gamma(f,g)(x)&\simeq& \|f\|_\ninf   \sup_{\bt\in (0,\infty)^2}\int^{c}_{-c} g\big(x-t_{2}y\big)  |y|^{-\frac{m-1}m} dy
	\end{eqnarray*}
	for a constant $c>0$. Since $|y|^{-\frac{m-1}m}\chi_{[-c, c]}$ is decreasing and integrable, the left hand side is bounded by $C \Vert f\Vert_{L^{\infty}} Mg(x).$ Thus, we get \eqref{est-p0}. Therefore, these estimates and interpolation give \eqref{estimateforfinitetypecurve} for $1< p,q\leq\infty$ and ${1}/{r}={1}/{p}+{1}/{q}$. The necessity part can be shown as before (see Remark \ref{rem-nec}).

		\subsubsection{Proof of  $(ii)$} 
		%
		Since $s_0=0$, using \eqref{near} we have
		\Be \label{mc} \mathcal{M}^{\gamma}(f,g)(x)=\sup_{ \bt\in (0,\infty)^2}\int f(x-t_{1}s)\,g  (x-t_{2}(c_\ast-s^{m}\phi(s)))\psi(s)ds,\Ee
		where $c_\ast=\gamma(s_0)$. 
		We make  change of variables $y=c_\ast-s^{m}\phi(s).$ Note $ |ds/dy|\sim |c_\ast-y|^{(1-m)/{m}}$. Thus, we get 
	\[
		\mathcal{M}^{\gamma}(f,g)(x)
		\lesssim \sup_{\bt\in (0,\infty)^2} \int_{|y-c_\ast|< c'}f(x-t_{1}s(y))\,g(x-t_{2}y)  |c_\ast-y|^{\frac{1-m}{m}} {dy} 
	\]
	for some $c'>0$.  Dyadic decomposition of  the integral away from $c_\ast$ gives
\[	\mathcal{M}^{\gamma}(f,g)(x) \lesssim \sum_{2^{-k}\le c'}  \mathcal{M}^{\gamma}_k(f,g)(x), \]
where
\Be
\label{mk}
\mathcal{M}^{\gamma}_k(f,g)(x) =2^{k\frac{m-1}m}  \sup_{\bt\in (0,\infty)^2} \int_{2^{-k} \le |y-c_\ast|< 2^{1-k} }f(x-t_{1}s(y))\,g(x-t_{2}y)   {dy} . 
\Ee

For $x\in \mathbb R$ and $l>0$, let us denote $I(x,l)=\{y\in \mathbb R:  |y-x|<l\}$.  
For $|h|\ge 1$, we  consider a maximal function 
\[ M^\ast_h f(x)=\sup_{t>0} \int_{I(h,2)}  f(x-ty)   {dy}.  \]  
If $|h|\lesssim 1$,  $ M^\ast_h f(x)\le C Mf(x)$ for a constant $C>0$. However,  if $|h|\gg 1,$  
$\int_{I(h,2)}  f(x-ty)   {dy}\le \int_{I(0,2|h|)}  f(x-ty)   {dy}\le 4h Mf(x)$. 
So, we have $\|M^\ast_h f\|_{L^{1,\infty}}\lesssim h\|f\|_1$. Interpolation with the trivial estimate $\|M^\ast_h f\|_{L^{\infty}}\lesssim \|f\|_\infty$  gives the following. 

\begin{lemma}\label{lemma1} 
Let $h\in \mathbb R$. Then, the estimate $\Vert M_{h}^{*} f\Vert_{L^p}\leq C(1+|h|)^{1/p}\Vert f\Vert_{L^{p}}$ holds for $1<p\le \infty$. 
		\end{lemma}

Surprisingly, the bound is sharp up to a constant multiple as can be seen by taking $f=\chi_{(0,1)}$. 
Indeed, note that $ M_{h}^{*} f(x)\sim 1/(1+h^{-1} x)$ for $x>0$ and $h\ge 2$. Thus, it follows that 
$\Vert M_{h}^{*} f\Vert_{L^p}/\|f\|_p\sim h^{1/p}$. 

Recalling \eqref{mk} and changing variables $h\to 2^{k} y$,  we have 
\[ \mathcal{M}^{\gamma}_k(f,g)(x) \lesssim 2^{-\frac{k}m}  \|f\|_\ninf  M^\ast_{2^kc_\ast} g(x)  .\]
So, by lemma \ref{lemma1} we have 
\[ \| \mathcal{M}^{\gamma}_k(f,g)\|_\qn \lesssim 2^{(\frac1q-\frac1m)k}  \|f\|_\ninf  \|g\|_\qn \]
for $1<q\le \infty$. 
\eqref{mk} and reversing the change of variables  $y=c_\ast-s^{m}\phi(s)$ yield  
$\mathcal{M}^{\gamma}_k(f,g)\lesssim 2^{-k/m} Mf(x) \|g\|_\infty$. Thus, we have  
$ \| \mathcal{M}^{\gamma}_k(f,g)\|_\pn \lesssim 2^{-k/m} \|f\|_\pn \|g\|_\ninf  $.  Therefore, 
interpolation gives 
\[ \| \mathcal{M}^{\gamma}_k(f,g)\|_\rn \lesssim 2^{(\frac1q-\frac1m)k}  \|f\|_\pn  \|g\|_\qn \]
for $1<p,q\le \infty$ and $1/r=1/p+1/q$. Therefore, we get \eqref{estimateforfinitetypecurve} for $1<p\leq\infty$, $m<q\leq\infty$ and ${1}/{p}+{1}/{q}<1$.

It now remains to prove the necessity part. To this end, we shall show that  \eqref{estimateforfinitetypecurve}  fails to hold if  $q\le m$ or ${1}/{p}+{1}/{q}\ge 1$. 
Failure of  \eqref{estimateforfinitetypecurve}  for ${1}/{p}+{1}/{q}\ge 1$ can be shown as before (see Remark \ref{rem-nec}). 
Thus, we only have to show that \eqref{estimateforfinitetypecurve}  fails  if  $q\le m$. To this end, we consider 
	\[ f=\chi_{[-1, 1]}, \quad  g=|x|^{-1/q}\log(1/|x|)^{-1}\chi_{[-1, 1]}(x).\]
	Then,  $g\in L^{q}$ for $1<q<\infty$. Since $\psi(s_0)\neq 0$, taking $t_1=C$ and $t_2=c_\ast$, we have 
	\[
		\mathcal{M}^{\gamma}(f,g)(x)
				\gtrsim  \int^{s_{0}+c}_{s_{0}-c}f\Big(x-\frac{x}{C}s\Big)g\Big(x-\frac{x}{c_\ast}(c_\ast- s^{m}\phi(s))\Big)ds.
	\]
	for some $c>0$. 
Taking $C>0$ large enough,  we see that 
	\Be  
		\mathcal{M}^{\gamma}(f,g)(x) \gtrsim  \int^{c_0}_{0}  (xs^{m})^{-1/q}\log^{-1}(\frac{1}{xs^{m}})~ds, \quad x\in (0, 1)
		\Ee
	for some positive constant  $c_0>0$.  Thus, changing variables $s\to x^{-1/m} s$ gives 
	\begin{eqnarray*}
		\mathcal{M}^{\gamma}(f,g)(x)
		\gtrsim x^{-1/m}\int^{c_0 x^{1/m}}_{0}   s^{-\frac mq} \log^{-1}(1/s)\, ds
	\end{eqnarray*}
	This shows $\mathcal{M}^{\gamma}(f,g)(x)=\infty$ for $x\in (0, 1)$ if $q\leq m$. 
	Hence, \eqref{estimateforfinitetypecurve} fails for $q\leq m$. 

\subsubsection{Proof of $(iii)$} 
Recall \eqref{mc} and \eqref{near}. By dyadic decomposition  away from $s=0$ after changing variables $s\to s+s_0$, we have  
		\[
		\mathcal{M}^{\gamma}(f,g)(x)
		\lesssim \sum_{2^{-k}\le c} \tilde{\mathcal{M}}^{\gamma}_k(f,g)(x),
			\]
	where 
	\[   \tilde{\mathcal{M}}^{\gamma}_k(f,g)(x)=  \sup_{\bt\in (0,\infty)^2}\int_{2^{-k}\le |s|<2^{1-k}}f(x-t_{1}(s+s_0))g\big(x-t_{2}(c_\ast-s^{m}\phi(s+s_0))\big)\,ds. \]
	Note that $\mathcal{M}^{\gamma}_k(f,g)(x)\lesssim   \sup_{t\in (0,\infty)}\int_{2^{-k}\le |s|<2^{1-k}}f(x-t_{1}(s+s_0))  ds  \|g\|_{L^\infty} $. Thus, changing variables $s\to 2^{-k}s$ gives 
	\Be 
	\label{pp}   \mathcal{M}^{\gamma}_k(f,g)(x)
	 \lesssim  2^{-k}  M^\ast_{2^ks_0} f(x) \|g\|_{L^\infty}.  
	 \Ee
	Note that $\mathcal{M}^{\gamma}_k(f,g)(x)\lesssim   2^{(m-1)k}  \|f\|_{L^\infty}  \sup_{t\in (0,\infty)}\int_{|y|<c2^{-mk}}g\big(x-t(c_\ast-y)\big)\,dy$ by change of variables $y=s^{m}\phi(s+s_0)$ since 
	$|dy/ds|\sim 2^{(1-m)k}$. 
	Thus,  similarly as above,  we get
	\Be
	\label{qq}  \mathcal{M}^{\gamma}_k(f,g)(x)
	\lesssim  2^{-k}  \|f\|_{L^\infty} M^\ast_{c2^{mk}} g(x) \Ee
	for some $c>0$. By those inequalities  we obtain
		\Be 
	\label{rpq} \mathcal{M}^{\gamma}_k(f,g)\|_\rn\lesssim 2^{(\frac 1p+\frac mq-1)k} \|f\|_\pn \|q\|_\qn 
	\Ee  
	for $1<p,q\le \infty$ and $1/r=1/p+1/q$. Indeed, \eqref{pp} and Lemma \ref{lemma1} give \eqref{rpq} for $r=p$,  $q=\infty$, and $1<p\le \infty$. 
	Besides,  \eqref{qq} and Lemma \ref{lemma1} give \eqref{rpq} for $r=q$,  $r=\infty$, and $1<q\le \infty$. 
	By interpolation  the estimate \eqref{rpq} follows on the above mentioned range of $p,q$. Consequently, we get \eqref{estimateforfinitetypecurve} if $1<p, q\le \infty$ and   ${m}/{q}+{1}/{p}<1$.


We prove that \eqref{estimateforfinitetypecurve} fails if ${m}/{q}+{1}/{p}\ge 1$ when $\gamma(s)=(s,\gamma_{2}(s))$ with $\gamma_{2}$  of type $m$ at $s_0$, $s_0\neq 0$, and $\gamma(s_0)\neq 0$. 
This can be shown by a slight modification of the argument in the proof of (the necessity part of) Theorem \ref{result2}. Since 
we already know \eqref{estimateforfinitetypecurve} fails if ${1}/{q}+{1}/{p}\ge 1$, we may assume $1\le q<\infty$. 

Since $\gamma_0$ is of type $m$ at $s_0$,  we have  $|\gamma_{2}(s_{0})-\gamma_{2}(s)|\sim |s-s_0|^m$ if   $s\in [s_{0}-c,s_{0}+c]$  with a constant $c>0$ small enough (see \eqref{near}).
		We consider 
		\[ f(x)=|x|^{-1/p}\log^{-\frac1p-\epsilon}(|x|^{-1})\chi_{|x|<1}, \quad  g(x)=|x|^{-\frac1q}\log^{-\frac1q-\epsilon}(|x|^{-1})\chi_{|x|<1},\]
		for an $\epsilon>0$ which we specify later. 
		 Then,  $f\in L^{p}(\mathbb{R})$ and $g\in L^{q}(\mathbb{R})$ for $1<p,q\le \infty$.  Recalling \eqref{max-} with $\gamma_1(s)=s$ and following the same argument as before, 	for $x\in (0,1)$ 	 
		 one can  see that 
			\[  \mathcal{M}^{\gamma}(f,g)(x)\gtrsim  \int^{c_0}_{0} (xs)^{-\frac1p-\frac mq}\log^{-\frac1p-\epsilon}\big(\frac{C_0}{xs}\big) 
			\log^{-\frac1q-\epsilon}\big(\frac{C_0}{xs^m}\big)
			ds \]
		for some constant $c_0, C_0>0$. Taking $c_0$ small enough and $C_0$ large enough, $\log\big(\frac{C_0}{xs}\big)\sim \log \big(\frac{C_0}{s}\big)$ and $\big(\frac{C_0}{xs^m}\big)\sim \log \big(\frac{C_0}{s}\big)$ if $s\in (0, c_0x).$  Thus, we have 
\[  \mathcal{M}^{\gamma}(f,g)(x)\gtrsim   x^{-\frac1p-\frac mq} \int^{c_0x}_{0} s^{-\frac1p-\frac mq}\log^{-\frac1p-\frac1q-2\epsilon}\big(\frac{C_0}{s}\big) 
			ds .\]		
Note that $1/p+1/q< 1$ since $q\neq \infty$, $1/p+ m/q\geq 1$, and $m>1$. Taking $\epsilon$ small enough such that 
$1/p+1/q+2\epsilon< 1$, we see  $\mathcal{M}^{\gamma}(f,g)(x)=\infty$ for $x\in (0,1)$ because ${m}/{q}+{1}/{p}\ge 1$.

\section{Multilinear generalizations}\label{multilinearversion} In this section we try to extend  the bilinear result to a multilinear setting. Let $m\ge 3$ and denote 
\[\bl=(\bl_1, \bl_2, \dots, \bl_m)\in [1,\infty)^m, \quad  \vec y= (y^1, y^2, \dots, y^m)\in (\mathbb R^{n})^m. \] 
We  consider the hypersurface $$S_{m}^{\bl}=\big\{\vec y\in \mathbb{R}^{mn}:\textstyle \sum^{m}_{j=1}|y^{j}|^{\bl_{j}}=1
\big\}$$
and an associated multilinear averaging operator given by 
\begin{eqnarray*}
\mathcal{A}_{m,\mathbf t}^{\bl }(\mathbf f)(x)=\int_{S_{m}^{\bl }}\prod^{m}_{j=1}f_{j}(x-t_{j}y^{j})~d\mu(\vec{y}),
\end{eqnarray*}
where $\mathbf {f}=(f_{1},\dots,f_{m})$, $\mathbf t=(t_1,\dots,t_{m})$, and $d\mu(\vec{y})$ denotes  the normalized surface measure on $S_{m}^{\bl }$. 
We study  
boundedness of the maximal operator 
\begin{eqnarray*}
\mathfrak{M}_{m}^{\bl }(\mathbf f)(x)=\sup_{\mathbf t>0}|\mathcal{A}_{m,\mathbf t}^{\bl }(\mathbf f)(x)|.
\end{eqnarray*}
Here, by $\sup_{\mathbf t>0}$ we mean $\sup_{t_1>0, \dots, t_m>0}$. We may assume that $f_{1},\dots,f_{m}$ are nonnegative.

\subsection{Trilinear maximal operators} We first study the case $m=3$. To state the result on  the (sub)trilinear operator $\mathfrak{M}^{\bl}_{3}$,   $\bl=(\bl_{1},\bl_{2},\bl_{3})\in [1,\infty)^3$, 
we introduce some notations. 
For a given $i\in \{1,2,3\}$ we set 
\[ 
	\frac{1}{p^{3}_{i}}=\frac{n}{\bl_{j_{2}}}+(1-\frac{n}{\bl_{j_{2}}})\Big(1-(\frac{1}{n}-\frac{1}{\bl_{j_{1}}})_{+}\Big)
	\]
	where $ j_{1}$, $ j_{2}$ are two distinct elements in the set $ \{1,2,3\}\setminus\{i\}$ such that $\bl_{j_{1}}\leq \bl_{j_{2}}$.  Particularly,  if $\bl_{j_1}\leq \bl_{j_2}\leq \bl_{j_3}$ and  $\bl_{j_{1}}\leq n$, we have $p^{3}_{j_2}=p^{3}_{j_3}=1$.  We also denote 
	\begin{eqnarray*}
	\mathcal{P}_{3}&=&\Big\{(x_{1},x_{2},x_{3})\in [0,1)^{3}:x_{1}+x_{2}+x_{3}<{(3n-1)}/{n}\Big\},
	\\
	\widetilde{\mathcal{P}}^{\bl}_{3}&=&\Big\{(x_{1},x_{2},x_{3})\in\mathcal{P}_{3}:x_{i}<1/{p^{3}_{i}},~~i=1,2,3\Big\}. 
\end{eqnarray*}

\begin{thm}\label{trilinearcase}
	Let $n\geq2$ and $1\leq p_{1},p_{2},p_{3}\leq\infty$ with ${1}/{p}=\sum^{3}_{j=1}{1}/{p_{j}}$. Then,   for $({1}/{p_{1}},{1}/{p_{2}},{1}/{p_{3}})\in\widetilde{\mathcal{P}}^{\bl }_{3}$ we have  the  estimate 
	\begin{eqnarray}\label{multilinearestimate2}
		\Vert \mathfrak{M}_{3}^{\bl }(\mathbf f)\Vert_{L^{p}(\mathbb{R}^{n})}\lesssim \prod^{3}_{j=1}\Vert f_{j}\Vert_{L^{p_{j}}(\mathbb{R}^{n})}. 
	\end{eqnarray}	
\end{thm}

Unlike Theorem \ref{Theorem1}, we do not expect that  the ranges in  Theorem \ref{trilinearcase}  are sharp (see Proposition \ref{thenecessarypart}).  One reason  that our argument is not enough to give the estimate 
in a sharp range is related to the lack of symmetry of the surface $S_{3}^{\bl }$, as can be seen,  in the estimates  $\eqref{eqn14}$ and $\eqref{eqn15}$

\begin{proof}[Proof of  Theorem \ref{trilinearcase}]
Again, the key idea  is to combine the slicing argument and 
boundedness of the bilinear maximal operator $\mathfrak{M}^{\bar\bl}_2$, $\bar\bl\in [1,\infty)^2$, which is previously shown. 
Let us set 
\[\textstyle \Phi(\vec y)=\sum^{3}_{j=1}|y^j|^{\bl_j}-1.\]
Using the identity \eqref{identity} for the hypersurface  $S^{\bl}_{3}=\{\vec y\in \mathbb{R}^{3n}:\Phi(\vec y)=0\}$, we get 
 \[
	\int_{S^{\bl}_{3}}\prod^{3}_{j=1}f_{j}(x-t_jy^j)~\frac{d\mu(\vec{y})}{|\nabla\Phi(\vec{y})|}=\int_{\mathbb{R}^{3n}}\prod^{3}_{j=1}f_{j}(x-t_jy^j) \delta ( \Phi(\vec y) )d\vec y. 
\]
Note that $|\nabla\Phi(\vec{y})|^{2}=\sum^{3}_{j=1}\bl^{2}_{j}|y^{j}|^{2(\bl_{j}-1)},$ so  $|\nabla\Phi(\vec{y})|> 0$ on $S^{\bl}_{3}$. 
Since $\bl_1, \bl_2, \bl_3\ge 1$,  we have $c_{\bl}\leq |\nabla\Phi(\vec{y})|\leq C_{\bl}<\infty$ on $S^{\bl}_{3}$ for some positive constants $c_{\bl},C_{\bl}$.  Thus, it follows that 
\Be
 \label{formula}
 \mathcal{A}^{\bl}_{3,\mathbf t}(\mathbf f)(x)\sim   \int_{\mathbb{R}^{3n}}\prod^{3}_{j=1}f_{j}(x-t_jy^j) \delta ( \Phi(\vec y) )d\vec y. 
\Ee

We set 
 \[ \nu_{a}(t)=(1-t)^{1/a}_+,
   \quad |\hat{y}^{k}|=\sum_{j\neq k} |y^{j}|^{\bl_{j}}.\]
We also denote $\vec{y}^{3}=(y^{1},y^{2}),$ $\vec{y}^{2}=(y^{1},y^{3}),$  and $\vec{y}^{1}=(y^{1},y^{3}).$  
Then, we claim that 
\Be
\label{3-3}
	\mathcal{A}^{\bl}_{3,\mathbf t}(\mathbf f)(x) \sim\int_{0\leq |\hat{y}^{k}|< 1}\prod_{j\neq k}f_{j}(x-t_{j}y^{j})\nu_{\bl_{k}}^{n-\bl_{k}}(|\hat{y}^{k}|)\mathfrak Af_{k}(x,t_{k}\nu_{\bl_{k}}(|\hat{y}^{k}|))d\vec{y}^{k}
\Ee
for  $k=1,2,3$. It is enough to show \eqref{3-3} for $k=3$. 
Denoting  $\vec{y}^{3}=(y^{1},y^{2})$,  we set $\Phi_{\vec{y}^3}(y^3)=\Phi(\vec{y})$ and observe
$\int_{\mathbb{R}^{3n}} F(\vec y) \delta ( \Phi(\vec y) )d\vec y=\int_{\mathbb{R}^{2n}} \int_{\mathbb{R}^{n}}   F(\vec y) \delta (\Phi_{\vec{y}^3}(y^3) ) dy^3 d\vec y^3$. 
Thus, using \eqref{formula} and the identity \eqref{identity}  for $\Phi_{\vec{y}^3}$, we  have
 \Be
 \label{formula1}
 	\mathcal{A}^{\bl}_{3,\mathbf t}(\mathbf f)(x)\sim\int_{0\leq |\hat{y}^{3}|< 1}\prod^{2}_{j=1}f_{j}(x-t_{j}y^{j})\int_{\Omega_{\vec{y}^3}}f_{3}(x-t_{3}y^{3})\frac{d\sigma_{\nu_{\bl_{3}}}(y^3)}{|\nabla\Phi_{\vec{y}^3}(y^3)|}~d\vec{y}^{3}, 
\Ee
 where $\Omega_{\vec{y}^3}=\Phi_{\vec{y}^3}^{-1}(0)\subset\mathbb{R}^{n}$ and $d\sigma_{\nu_{\bl_{3}}}$ is the surface measure on the sphere $\Omega_{\vec{y}^3}$.
 A computation shows $|\nabla\Phi_{\vec{y}^{3}}(y^3)|=\bl_{3}\nu^{\bl_3-1}_{\bl_{3}}(|\hat{y}^3|)$. Note that $\Omega_{\vec{y}^3}\subset\mathbb{R}^{n}$ is an $(n-1)$-dimensional sphere of radius $\nu_{\bl_{3}}(|\hat{y}^3|)$. Thus, by scaling we get 
 \eqref{3-3} for $k=3$. 

To prove Theorem \ref{trilinearcase},  we may assume, without loss of generality, that \Be\bl_1\leq \bl_2\leq \bl_3.\Ee
We consider the following cases, separately:  
\begin{align*}
&{(\mathrm A)}: \bl_{3}\leq n,  \quad\qquad \ \ (\mathrm B):\bl_2\leq n<\bl_3,  
\\
& (\mathrm C): \bl_1\leq n<\bl_2,
\quad \ (\mathrm D):  n<\bl_1.
\end{align*}

\subsection*{Case $(\mathrm A):\bl_3\leq n$} 
Since $\bl_1, \bl_2, \bl_{3}\leq n$, $\nu_{\bl_{k}}^{n-\bl_{k}}(|\hat{y}^{k}|)\le 1$.  By \eqref{3-3}  it follows that  
\Be
\label{mms}
\textstyle
\mathcal{A}^{\bl}_{3,\mathbf t}(\mathbf f)(x)\lesssim  \mathcal M_{s}f_{k}(x) \prod_{j\neq k} Mf_{j}(x)
\Ee
for $k=1,2,3$. Using H\"older's inequality and  $L^p$ boundedness of 
$M$ and $\mathcal M_s$,  we obtain  \eqref{multilinearestimate2}   for $\frac{n}{n-1}<p_{k}\leq\infty$ and $1<p_{j}\leq\infty$, $j\neq k$.

Now, interpolation gives  $L^{p_{1}}\times L^{p_{2}}\times L^{p_{3}}\rightarrow L^{p}$ bound on   $\mathfrak{M}_{3}^{\bl}$ for $({1}/{p_{1}},{1}/{p_{2}},{1}/{p_{3}})\in\mathcal{P}_{3}$.

\medskip

We consider the case  $(\mathrm D)$ before  the  cases $(\mathrm B)$ and $(\mathrm C)$.  Since  $\sum^{3}_{j=1}|y^j|^{\bl_j}=1$ in \eqref{formula}, it is clear that 
\Be 
\label{primary}
 \mathcal{A}^{\bl}_{3,\mathbf t}(\mathbf f)(x) \lesssim \sum^{3}_{l=1}\mathcal{T}^{l}_{\mathbf t}(\mathbf f)(x) :=\sum^{3}_{l=1}\int_{|y^l|^{a_l}\le 1/2}\prod^{3}_{j=1}f_{j}(x-t_jy^j) \delta ( \Phi(\vec y) )d\vec y.
 \Ee
 Using the identity \eqref{identity}  for $\Phi_{\vec{y}^l}$ (cf. \eqref{formula1}), we see that 
\Be
\label{333}
\textstyle\mathcal{T}^{l}_{\mathbf t}(\mathbf f)(x)  \lesssim  \sum^{\infty}_{k=2}\mathcal{T}^{l, k}_{\mathbf t}(\mathbf f)(x), 	
\Ee
where 
	\[  \mathcal{T}^{l,k}_{\mathbf t}(\mathbf f)(x)=  \int_{2^{-k}\leq 1- |\hat{y}^{l}|\leq 2^{1-k}}\prod_{j\neq l}f_{j}(x-t_{j}y^{j})\nu_{\bl_{l}}^{n-\bl_{l}}(|\hat{y}^{l}|)\,\mathfrak Af_{l}(x,t_{l}\nu_{\bl_{l}}(|\hat{y}^{l}|))d\vec{y}^{l}.\]

\subsection*{Case $(\mathrm D): n<\bl_1$}

We first consider $\mathcal{T}^{3}_{\mathbf t}$. 	
Since  $\nu_{\bl_{3}}^{n-\bl_{3}}(|\hat{y}^{3}|)=(1-|\hat{y}^{3}|)^{(n-\bl_{3})/{\bl_{3}}}\lesssim 2^{{k(\bl_{3}-n)}/{\bl_{3}}}$ in the above integral,  it follows that 
\begin{eqnarray}\label{eqn14}
	\mathcal{T}_{\mathbf t}^{3,k}(\mathbf f)(x)\lesssim2^{\frac{k(\bl_{3}-n)}{\bl_{3}}}Mf_{1}(x)Mf_{2}(x)\mathcal M_{s}f_{3}(x)
\end{eqnarray}
for  $k\geq 2$. Thus, the operator $\mathbf f \to \sup_{\bt>0} \mathcal{T}^{3,k}_{\mathbf t}(\mathbf f)$ is bounded from $L^{p_{1}}\times L^{p_{2}}\times L^{p_{3}}$ to $L^{p}$ for $1<p_{1},p_{2}\leq\infty$ and ${n}/(n-1)<p_{3}\leq\infty$ with its operator norm dominated by $C2^{{k(\bl_{3}-n)}/{\bl_{3}}}$.

Set $\Phi_{{y}^3}(\vec y^3)=\Phi(\vec{y})$. Recalling \eqref{primary} and using \eqref{identity}  for $\Phi_{{y}^3}$,   as before
we have 
\begin{eqnarray*}
	\mathcal{T}_{\mathbf t}^{3,k}(\mathbf f)(x)\sim\!\int_{\mathrm A_k}f_{3}(x-t_{3}y^{3})\nu_{1}^{\alpha}(|y^{3}|^{\bl_3})\!\!\int_{S_{2}^{\bar\bl}}\prod^{2}_{j=1}f_{j}\big(x-t_{j}\nu_{\bl_{j}}(|y^{3}|^{\bl_3})y^{j}\big)\frac{d\mu(\vec{y}^3)}{|\nabla \Phi_{y^{3}}|}dy^{3},
\end{eqnarray*}
where  $\mathrm A_k=\{y^3: |y^{3}|\sim 2^{\frac {-k}{\bl_{3}}}\}$,  $\alpha=({n}/{\bl_{1}})+({n}/{\bl_{2}})-1$,  $\bar \bl=(\bl_1, \bl_2)$,  $d\mu$ is the surface measure on  $S_{2}^{\bar\bl}$.
Noting that 
\begin{eqnarray}\label{Gradient}
	|\nabla \Phi_{y^{3}}(\vec{y}^3)|^2=\sum^{2}_{j=1}\bl^{2}_{j}\nu_{\bl_{3}}^{{2(\bl_{j}-1)}}(|y^{3}|^{\bl_3})|y^{j}|^{2\bl_{j}-2}, 
	\end{eqnarray}
we have$|\nabla \Phi_{y^{3}}(\vec{y}^3)|\geq \min\{\bl_{2}(\frac{7}{16})^{\frac{\bl_{1}-1}{\bl_{1}}},\bl_{1}(\frac{7}{16})^{\frac{\bl_{2}-1}{\bl_{2}}}\}$. Also,  $\nu_{1}^{\alpha}(|y^{3}|^{\bl_3})=(1-|y^{3}|^{\bl_{3}})_{+}^{\alpha}$ is bounded above by some fixed constant even if $\alpha<0$. Consequently, it follows that
\begin{eqnarray}\label{eqn15}
	\mathcal{T}_{\mathbf t}^{3,k}(\mathbf f)(x)\lesssim  2^{-kn/\bl_{3}} \mathfrak{M}^{(\bl_1,\bl_2)}(f_{1},f_{2})(x) Mf_{3}(x).
\end{eqnarray}
Thus,  using Theorem \ref{Theorem1},  we see that $\mathcal{T}_{\mathbf t}^{3,k}$ is bounded from   $L^{p_{1}}\times L^{p_{2}}\times L^{p_{3}}$ to $L^{p}$ for $1<p_{3}\leq\infty$ and  ${1}/{p_{1}}<1-({1}/{n}-{1}/{\bl_{2}}), {1}/{p_{2}}<1-({1}/{n}-{1}/{\bl_{1}})$ with the operator norm dominated by $C2^{{-nk}/{\bl_{3}}}$. 


Now, using the estimates resulting from  $\eqref{eqn14}$ and $\eqref{eqn15}$,  via  interpolation  we see that the operator $\mathbf f \to \sup_{\bt>0} \mathcal{T}^{3,k}_{\mathbf t}(\mathbf f)$ is bounded from $L^{p_{1}}\times L^{p_{2}}\times L^{p_{3}}$ to $L^{p}$ with the operator norm dominated by $2^{-\epsilon k}$ for some $\epsilon>0$  
provided that 
\Be
\label{exp}
\begin{aligned}
\frac{1}{p_{1}}&<\frac{n}{\bl_{3}}+(1-\frac{n}{\bl_{3}})\big(1-(\frac{1}{n}-\frac{1}{\bl_{2}})\big),
\\
\frac{1}{p_{2}}&<\frac{n}{\bl_{3}}+(1-\frac{n}{\bl_{3}})\big(1-(\frac{1}{n}-\frac{1}{\bl_{1}})\big), 
\\
\frac{1}{p_{3}}&<\frac{n}{\bl_{3}} \frac{n-1}{n} + (1-\frac{n}{\bl_{3}}) =\frac{\bl_{3}-1}{\bl_{3}}.
\end{aligned} 
\Ee
Hence,   recalling \eqref{333}, we see that  $\mathbf f \to \sup_{\bt>0} \mathcal{T}^{3}_{\mathbf t}(\mathbf f)$ bounded from $L^{p_{1}}\times L^{p_{2}}\times L^{p_{3}}\rightarrow L^{p}$ as long as \eqref{exp} holds. 
Using the same argument,  for $i=1,2$ we also see $\mathbf f \to \sup_{\bt>0} \mathcal{T}^{i}_{\mathbf t}(\mathbf f)$ is bounded from $L^{p_{1}}\times L^{p_{2}}\times L^{p_{3}}\rightarrow L^{p}$
if  $ 1/{p_{i}}<({\bl_i-1})/{\bl_i}$, $1/p_{3-i}<1/p_{3-i}^{3} $ and $1/{p_{3}}<1/p^{3}_{3} $.  Note that 
$({\bl_i-1})/{\bl_i}<1/p^{3}_{i}$ for $i=1,2,3$.  Therefore, by \eqref{primary} we have the estimate \eqref{multilinearestimate2} for $({1}/{p_{1}},{1}/{p_{2}},{1}/{p_{3}})\in\widetilde{\mathcal{P}}^{\bl }_{3}$.

\subsection*{Case $(\mathrm B):\bl_2\leq n<\bl_3 $} 

Since $\bl_1, \bl_2\le n$, using \eqref{3-3}
 we have \eqref{mms} for $k=1,2$. 
Invoking $L^{p}$-boundedness  of the Hardy--Littlewood maximal operator $M$ and the spherical maximal operator $\mathcal{M}_{s}$, 
we see that \eqref{multilinearestimate2} holds if $\frac{n}{n-1}<p_{1}\leq \infty$ and $ 1<p_{2},p_{3}\leq\infty$, and if  $\frac{n}{n-1}<p_{2}\leq \infty$ and $ 1<p_{1},p_{3}\leq\infty$. On the other hand, from~\eqref{3-3} we have 
\begin{eqnarray*}
	\mathcal{A}^{\bl}_{3,\mathbf t}(\mathbf f)(x)\sim\int_{0\leq |\hat{y}^{3}|< 1}\prod^{2}_{j=1}f_{j}(x-t_{j}y^{j})\nu_{\bl_{3}}^{n-\bl_{3}}(|\hat{y}^{3}|)\,\mathfrak Af_{3}(x,t_{3}\nu_{\bl_{3}}(|\hat{y}^{3}|))~d\vec{y}^{3}.
\end{eqnarray*}
Since $\bl_3>n$,  $\nu_{\bl_{3}}^{n-\bl_{3}}(|\hat{y}^{3}|)$ is no longer bounded. We further decompose the above integral as in Case $(\mathrm D)$.
Indeed, we have 
\[\mathcal{A}^{\bl}_{3,\mathbf t}(\mathbf f)(x)\sim \sum^{\infty}_{k=1}\mathcal{T}^{k}_{\mathbf t}(\mathbf f)(x),\]
where 
\[\mathcal{T}^{k}_{\mathbf t}(\mathbf f)(x):=\int_{2^{-k}\leq 1- |\hat{y}^{3}|\leq 2^{1-k}}\prod^{2}_{j=1}f_{j}(x-t_{j}y^{j})\nu^{n-\bl_{3}}_{\bl_{3}}(|\hat{y}^{3}|)\mathfrak Af_{3}(x,t_{3}\nu_{\bl_{3}}(|\hat{y}^{3}|))~d\vec{y}^{3}.
\]
For $\mathcal{T}^{1}_{\mathbf t}(\mathbf f)$, observe that  $\nu^{n-\bl_{3}}_{\bl_{3}}(|\hat{y}^{3}|)=(1-|\hat{y}^{3}|)_{+}^{(n-\bl_{3})/{\bl_{3}}}$ is bounded above by a fixed constant.
Hence, 
\begin{eqnarray}\label{eqn13}
	\mathcal{T}^{1}_{\mathbf t}(\mathbf f)(x)\lesssim Mf_{1}(x)Mf_{2}(x)\mathcal M_{s}f_{3}(x).
\end{eqnarray}
So, \eqref{multilinearestimate2} holds for $\mathcal{T}^{1}_{t}$ for $1<p_{1},p_{2}\leq \infty$ and $n/(n-1)<p_{3}\leq\infty$.

 On the other hand, for $k\geq2$ the operator $\mathcal{T}^{k}_{t}$  satisfies  \eqref{eqn14} and \eqref{eqn15}. Since $\bl_1, \bl_2\leq n$ and $n< \bl_3$, using 
 $L^p$ boundedness of the Hardy--Littlewood maximal and 
 spherical functions and applying  interpolation with $\theta:=n/\bl_{3}$, we see that 
 $\mathbf f \to \sup_{\bt>0} \mathcal{T}^{k}_{\mathbf t}(\mathbf f)$ is bounded from $L^{p_{1}}\times L^{p_{2}}\times L^{p_{3}}$ to $L^{p}$ with its operator norm dominated by $C2^{-\epsilon k}$ for some $\epsilon>0$  
provided that 
\[
	\frac{1}{p_{1}}, \frac{1}{p_{2}}< \theta \times 1+(1-\theta)\times 1, \quad 
	\frac{1}{p_{3}}<\theta\times \frac{n-1}{n} + (1-\theta)\times 1 =\frac{\bl_{3}-1}{\bl_{3}}.
\]

Finally, we combine all the three estimates of $\mathfrak{M}^{\bl}_{3}$ and apply multilinear interpolation to get the desired boundedness. 

\textbf{Case $(\mathrm C)$} can be handled in the similar manner. We omit the detail.
\end{proof}

\subsection{Multilinear maximal bounds for $m\geq4$} As in  the trilinear case, using the slicing argument, we  get
\Be\label{I}
	\mathcal{A}^{\bl }_{m, \mathbf t}(\mathbf f)(x)\sim\int_{ |\hat{y}^{l}|< 1}\prod_{j\neq l}f_{j}(x-t_{j}y^{j})\nu_{\bl_l}^{n-\bl_{l}}(|\hat{y}^{l}|)\mathfrak Af_{l}(x,t_{l}\nu_{\bl_l}(|\hat{y}^{l}|)) ~d\vec{y}^{l},
\Ee
for $l=1,\dots, m$, where  $\nu_{\bl_l}(t)=(1-t)_{+}^{1/\bl_{l}}$, $|\hat{y}^{l}|=\sum_{j\neq l }|y^{j}|^{\bl_{j}}$ and $d\vec{y}^{l}=dy^{1}\cdots dy^{l-1}$ $dy^{l+1}\cdots dy^{m}$ missing $dy^{l}$. 
If $\bl_1, \bl_2, \dots, \bl_m\le n$, 
we get 
\[
	\textstyle \mathfrak{M}_{m}^{\bl }(\mathbf f)(x)\lesssim \prod^{m}_{j=1,j\neq i}Mf_{j}(x)\mathcal M_{s}f_{i}(x), \quad i=1, 2,\dots, m-1. 
\]
By boundedness of the Hardy-Littlewood and  spherical maximal operators and interpolation we have  
\Be	\label{multilinearestimate22}
		\Vert \mathfrak{M}_{m}^{\bl }(\mathbf f)\Vert_{L^{p}(\mathbb{R}^{n})}\lesssim \prod^{m}_{j=1}\Vert f_{j}\Vert_{L^{p_{j}}(\mathbb{R}^{n})}. 
	\Ee
 provided that $1<{p_{1}},\dots,{p_{m}}\le \infty$ and $\sum^{m}_{i=1}1/p_{i}<(mn-1)/{n} .$

If $\bl_i>n$ for some $i$, $\nu_{\bl_i}(|\hat{y}^{i}|)^{n-\bl_{i}}$ in \eqref{I} is not bounded any longer. Nevertheless, 
following the previous argument (see, Case $(\mathrm D)$ in the proof of  Theorem \ref{trilinearcase}),  one can obtain maximal bounds for $\mathfrak{M}_{m}^{\bl }$ by decomposing the operator dyadically away from the the set $|\hat{y}^{i}|=1$.   
Indeed,  recall that boundedness of the bilinear maximal operators $\mathfrak{M}^{\bar \bl}$, $\bar \bl\in [1,\infty)^2$ is used to prove  the trilinear estimates in the proof of Theorem \ref{trilinearcase}(e.g., see \eqref{eqn15}).  
This idea can be further generalized so that one can obtain the estimate for  the $m$ linear maximal operator $\mathfrak{M}_{m}^{\bl}$ while assuming boundedness of  the $(m-1)$ linear maximal  operators $\mathfrak{M}_{m-1}^{\bar\bl}$, $\bar{\bl}\in [1,\infty)^{m-1}$.  So, one can extend the trilinear result to any multilinear maximal operator in inductive manner. 
 This can be done by routine adoption of the proof of Theorem \ref{trilinearcase}. We leave it to the interested reader. However, the consequent result is far from being sharp. 
 
 Finally, we discuss necessary conditions on the exponents $p_1, \dots, p_m$  for the estimate \eqref{multilinearestimate22} to hold.

\begin{prop}\label{thenecessarypart}
	Let $n\geq2$, $1\leq p_{1},p_{2},\dots,p_{m}\leq\infty$, and ${1}/{p}=\sum^{m}_{j=1}{1}/{p_{j}}$. Then  \eqref{multilinearestimate22} holds only if $\sum^{m}_{i=1}1/p_{i}\le (mn-1)/{n}$ and 
	\Be\label{2nd} 1/p_{i}<1-\big({1}/{n}-\textstyle\sum^{m}_{j=1,j\neq i}\frac{1}{\bl_{j}}\big)_{+},\quad i=1,2,\dots,m.\Ee
	
\end{prop}

\begin{proof}[Proof of Proposition \ref{thenecessarypart}]
One can easily show that  the condition $\sum^{m}_{i=1}1/p_{i}\le (mn-1)/{n}$ is necessary for 
\eqref{multilinearestimate22} using \eqref{I} and the characteristic functions of balls of radius $\delta$ (see Section \ref{sec:2.2}).  
We only show \eqref{2nd}.

Consider $f_{2}= f_{3}=\cdots=f_{m}=\chi_{B(0,4)}$ and $f_{1}=\chi_{B(0,C\delta)}$ for some constant $C>2$ and $\delta<1/(2Cm)$. 
Let  
\[  R(\delta)=\{\vec{y}^{1}\in (\mathbb{R}^{n})^{m-1}:|y^{j}|^{\bl_{j}}\leq\delta, \quad j=2,3,\dots, m\}.\] Note that  volume of the rectangle $R(\delta)$ is  $\sim \delta^{\sum^{m}_{j=2}{n}/{\bl_j}}$. 
We now use  \eqref{I} with $l=1$ and  observe that  $\nu_{\bl_{1}}^{n-\bl_{1}}(|\hat{y}^{1}|)\sim1$  for $\vec{y}^{1}\in R(\delta)$. Therefore, for $1\leq |x|\leq 2$, we see that 
\[
 \mathfrak M^{\bl}_{m}(\mathbf f)(x)
 	\gtrsim  \sup_{\bt>0} \int_{R(\delta)}\prod^{m}_{j=2}f_{j}(x-t_{j}y^{j})\,\mathfrak Af_{1}(x,t_{1}\nu_{\bl_1}(|\hat{y}^{1}|))\,d\vec{y}^{1}
 	\gtrsim  \delta^{\sum^{m}_{j=2}\frac{n}{\bl_j}}\delta^{n-1}.
\]
Hence, the estimate  \eqref{multilinearestimate22} implies $
\delta^{(n-1)+\sum^{m}_{j=2}n/{\bl_j}}\lesssim \delta^{\frac{n}{p_{1}}}.$
Taking $\delta\rightarrow0$ gives \eqref{2nd} for $i=1$.  Interchanging the role of the functions, we similarly obtain  \eqref{2nd} for $2\leq i\le  m$.	
\end{proof}

\section*{ Acknowledgements}  
This work was  supported by the  South Korea NRF grant no.  2022R1A4A1018904 (Lee $\&$ Shuin) and  BK21  Post doctoral fellowship of Seoul National University (Shuin).

\end{document}